\documentclass[12pt]{article}
\usepackage{amsmath,amssymb,amsthm}
\usepackage{comment}
\usepackage{color}
\numberwithin{equation}{section}
\newtheorem{thm}{Theorem}[section]
\newtheorem{prop}[thm]{Proposition}
\newtheorem{cor}[thm]{Corollary}

\newtheorem{rem}[thm]{Remark}
\newtheorem{lem}[thm]{Lemma}

\newtheorem{assum}[thm]{Assumption}

\def\E{{\cal E}}
\def\F{{\cal F}}
\def\N{{\cal N}}
\def\R{{\mathbb R}}
\def\Pp{{\mathbb P}}
\def\Ee{{\mathbb E}}
\def\d{{\rm d}}

\allowdisplaybreaks
\def\address#1#2{\begingroup
\noindent\parbox[t]{16cm}{%
\small{\scshape\ignorespaces#1}\par\vskip1ex
\noindent\small{\itshape E-mail address}%
\/: #2\par\vskip4ex}\hfill%
\endgroup}%

\title{Upper Rate Functions of Brownian Motion Type for Symmetric Jump Processes}
\author{Yuichi Shiozawa\quad Jian Wang}

\begin{document}
\maketitle
\begin{abstract}
Let $X$ be a symmetric jump process on $\R^d$ such that
the corresponding jumping kernel $J(x,y)$ satisfies
$$J(x,y)\le \frac{c}{|x-y|^{d+2}\log^{1+\varepsilon}(e+|x-y|)}$$
for all $x,y\in\R^d$ with $|x-y|\ge1$ and some constants $c,\varepsilon>0$.
Under additional mild assumptions on $J(x,y)$ for $|x-y|<1$,
we show that $C\sqrt{r\log \log r}$ with some constant $C>0$ is an upper rate function of the process $X$,
which enjoys the same form as that for Brownian motions.
The approach is based on heat kernel estimates of large time for the process $X$.
As a by-product, we also obtain two-sided heat kernel estimates of large time
for symmetric jump processes whose jumping kernels are comparable to
$$\frac{1}{|x-y|^{d+2+\varepsilon}}$$
for all $x,y\in\R^d$ with $|x-y|\ge1$ and some constant $\varepsilon>0$.
\end{abstract}
 \vskip 2cm

\section{Introduction and Main Results}
In this paper, we are concerned with upper rate functions,
which are a quantitative expression 
of conservativeness,
for a class of symmetric jump processes on ${\mathbb R}^d$.
In particular, we investigate conditions on jumping kernels
such that the corresponding upper rate functions are of the iterated logarithm type.

It is well known that by Kolmogorov's test (see, e.g., \cite[4.12]{IM}),
the function $R(t)=\sqrt{ct\log\log t}$ with constant $c>0$ is an upper rate function
for the standard Brownian motion on $\R^d$
if and only if  $c>2$.
This fact immediately implies Khintchine's law of the iterated logarithm.
Similar results of this type are true even for a large class of  L\'evy processes.
For example, earlier Gnedenko \cite{Gn43} (see also \cite[Proposition 48.9]{Sa13}) showed that
if a L\'evy process $X=(\{X_t\}_{t\ge0}, \Pp)$ on $\R$ satisfies
$\Ee X_1=0$ and $\Ee X_1^2<\infty$, then
$$\limsup_{t\to\infty} \frac{|X_t|}{\sqrt{t\log\log t}}=(\Ee X_1^2)^{1/2}, \quad \text{a.s.}$$
Sirao \cite{Si53} also obtained analogous  results in terms of
integral tests on the distribution function of $X$.
We note that such results as \cite{Gn43,Si53} do not hold in general
for L\'evy processes with the infinite second moment, for instance,
symmetric $\alpha$-stable processes with $\alpha\in (0,2)$  (see \cite{Kh38} or  \cite[Theorem 2.1]{Sa01}).

The purpose of this paper is to establish upper rate functions of the form $\sqrt{t\log\log t}$
for a class of non-L\'evy symmetric jump processes generated
by regular Dirichlet forms on $L^2({\mathbb R}^d;\d x)$, which we introduce later.
Let $J(x,y)$ be a non-negative measurable function on $\R^d\times \R^d$, and set
\begin{align*}
{\cal D}&=\left\{f\in L^2({\mathbb R}^d;\d x) \Big|\iint_{x\neq y} (f(y)-f(x))^2J(x,y)\,\d x\,\d y<\infty\right\},\\
\E(f,f)&=\iint_{x\neq y} (f(y)-f(x))^2J(x,y)\,\d x\,\d y, \quad f\in {\cal D}.
\end{align*}
Throughout this paper, we always impose the following
\begin{assum}\label{assum:jump}\rm The function $J(x,y)$ satisfies
\begin{itemize}
\item[(i)] $J(x,y)=J(y,x)$ for all $x\neq y$;
\item[(ii)] there exist constants
$0<\kappa_1\le \kappa_2<\infty$
and $0<\alpha_1\le\alpha_2<2$ such that for all $x,y\in \R^d$ with $0<|x-y|<1$,
\begin{equation}\label{A:jump-kernel}
\frac{\kappa_1}{|x-y|^{d+\alpha_1}}\leq J(x,y)\leq \frac{\kappa_2}{|x-y|^{d+\alpha_2}};\end{equation}
\item[(iii)]\begin{equation}\label{e:01}\sup_{x\in\R^d} \int_{\{|x-y|\ge1\}}J(x,y)\,\d y<\infty.
\end{equation}
\end{itemize}
\end{assum}
\noindent Denote by $C_c^{\rm lip}(\R^d)$ the set of Lipschitz continuous functions on $\R^d$ with compact support.
Let $\F$ be the closure of $C_c^{\rm lip}(\R^d)$ with respect to the norm
$\|f\|_{{\cal E}_1}:=\sqrt{\E(f,f)+\|f\|_2^2}$ on ${\cal D}$.
Then it is easy to check that
the bilinear form $(\E,\F)$ is a symmetric regular Dirichlet form on $L^2(\R^d;\d x)$, see e.g.\ \cite[Example 1.2.4]{FOT11}.
The function $J(x,y)$  is called the jumping kernel corresponding to $(\E,\F)$.
Associated with the regular Dirichlet form $(\E,\F)$ is a symmetric Hunt process
$X=(\{X_t\}_{t\ge0}, \{\Pp^x\}_{x\in\R^d\setminus \N})$ with state space $\R^d\setminus \N$,
where $\N\subset \R^d$ is a properly exceptional set for $(\E,\F)$.

\ \

The main result is as follows.

\begin{thm}\label{main} Let $X=(\{X_t\}_{t\ge0}, \{\Pp^x\}_{x\in\R^d\setminus \N})$
be the symmetric Hunt process generated by the regular Dirichlet form $(\E,\F)$ as above.
Let $J(x,y)$ be the jumping kernel corresponding to $(\E,\F)$. Suppose that
\begin{equation}\label{e:second}
\sup_{x\in \R^d}\int_{{\mathbb R}^d} |x-y|^2J(x,y)\,\d y <\infty.
\end{equation}
Then, we have the following two statements.
\begin{itemize}
\item[$(1)$]
If there exist positive constants $c$ and $\varepsilon$ such that for any $x,y\in\R^d$ with $|x-y|\ge 1$,
$$J(x,y)\le \frac{c}{|x-y|^{d+2}\log^{1+\varepsilon}(e+|x-y|)},$$
then there exists a constant $C_0>0$ such that for all $x\in \R^d \backslash \N$,
\begin{equation}\label{e:rate}
\Pp^x(|X_t-x|\le C_0\sqrt{ t\log\log t} \text{ for all sufficiently large }t)=1.
\end{equation}

\item[$(2)$] If there exists a positive constant $c$ such that for any $x,y\in\R^d$ with $|x-y|\ge 1$,
$$J(x,y)\le \frac{c}{|x-y|^{d+2}},$$ then there exists a constant $c_0>0$ such that for all $x\in \R^d \backslash \N$,
$$\Pp^x(|X_t-x|\le c_0\sqrt{ t\log\log t} \text{ for all sufficiently large } t)=0.$$
\end{itemize}
\end{thm}

The condition \eqref{e:second} implies that the jumping kernel of $X$ has the finite second moment.
\eqref{e:rate} indicates that the function $C_0\sqrt{t\log\log t}$ is
the so-called upper rate function of the process $X$,
which describes the forefront of the process $X$.
As we mentioned before, $\sqrt{(2+\varepsilon)t\log\log t}$ with $\varepsilon>0$ is an upper rate function
for the standard Brownian motion on $\R^d$.
Therefore, Theorem \ref{main} shows that
if the jumping kernel of $X$ satisfies the condition as in Theorem \ref{main} (1),
then $X$ enjoys upper rate functions of the Brownian motion type.

According to the results of \cite{Gn43,Si53},
we believe that $C\sqrt{t\log\log t}$ with some large constant $C>0$ should be
an upper rate function for all symmetric jump processes
with finite second moments;
however, we do not know how to prove this at this stage.
Here it should be noted that the arguments of \cite{Gn43,Si53} heavily depend
on the characterization of L\'evy processes (see \cite[Sections 2 and 3]{Sa01} for more details),
while in the present setting such characterization is not available.
To overcome this difficulty,
we prove Theorem \ref{main} by using heat kernel estimates.
The idea of obtaining rate functions via heat kernel
estimates has appeared in the literatures before, see \cite{SW17} and the references therein.
There are a few differences and difficulties in the present paper, which require some new ideas  and non-trivial arguments.
\begin{itemize}
\item[(1)]
For symmetric jump processes of variable order (see \eqref{A:jump-kernel}),
it seems impossible to present two-sided estimates
for the associated heat kernel, see \cite{BBCK09} for details.
Instead of this approach, here we turn to consider the heat kernel estimate
for large time, which is enough to yield the rate function of the process.

\item[(2)] There are a lot of works on heat kernel estimates
for symmetric jump processes on $\R^d$ generated by non-local symmetric Dirichlet forms,
see \cite{BBCK09, BGK09, CKK08, CK10, CKK11, Fo09} and the references therein.
However, there seems no study on the heat kernel estimates
when the jumping kernel has the finite second moment (even with precise algebraic decay).
Despite this, we can establish two-sided heat kernel estimates of large time
for symmetric jump processes whose jumping kernels are comparable to
$|x-y|^{-(d+2+\varepsilon)}$ for all $x,y\in\R^d$ with $|x-y|\ge1$ and some constant $\varepsilon>0$
(Corollary \ref{c:two1}).
We can also obtain nice
upper bounds of heat kernel estimates for processes
whose jumping kernel involves the logarithmic factor (Theorem \ref{thm-upper-bound}).
\end{itemize}

By analogy with Brownian motions,
one may guess that in order to prove Theorem \ref{main},
it suffices to get Gaussian type upper bound estimates for the heat kernel.
However, as far as we have discussed in this paper,
such upper bounds are only true for some interval of large time, not for all large time.
This is quite different from the Brownian motion case, and so we need further considerations
on the heat kernel bounds
(Theorem \ref{thm-upper-bound} and the proof of Theorem \ref{main} in the last section).

Bass and Kumagai \cite{BaKu08} proved
the convergence to symmetric diffusion processes
of continuous time random walks on ${\mathbb Z}^d$ with unbounded range.
In particular, they assumed the uniform finite second moment condition on conductances
similar to \eqref{e:second} on jumping kernels, see \cite[(A3) in p.\ 2043]{BaKu08}.
For the proof of the convergence result,
they obtained sharp on-diagonal heat kernel estimates, H\"{o}lder regularity of parabolic functions and Harnack inequalities.
Our result can be regarded as an another approach
to get the diffusivity of symmetric jump processes with jumping kernels having the finite second moment.

\ \

The reminder of this paper is arranged as follows.
In the next section, we recall some known results for heat kernel of the process $X$, and then present
related  assumptions used in our paper. Section \ref{section3} is devoted to establish upper bounds and lower bounds of heat kernel for large time. In particular,
Theorems \ref{thm-upper-bound} and \ref{T:lower} are interesting of their own.
Then the proof of  Theorem \ref{main} will be presented in the last section.

\ \

For any two positive measurable functions $f$ and $g$,
$f\asymp g$ means that there is a constant $c>1$ such that $c^{-1} f\le g\le cf$.

\section{Known results and assumptions}
Recall that $X=(\{X_t\}_{t\geq 0}, \{\Pp^x\}_{x\in \R^d\setminus \N})$ is
the Hunt process associated with $(\E,\F)$,
which can start from any point in $\R^d\setminus \N$.
Let $P(t,x,\d y)$ be the transition probability of
$X$. The transition semigroup $\{P_t,t\ge0\}$ of $X$ is defined
for $x\in \R^d\setminus \N$ by
$$P_tf(x)=\Ee^x (f(X_t))=\int_{\R^d} f(y)\,P(t,x, \d y),\quad f\ge 0, t\ge 0.$$

The following result has been proved in \cite[Theorem 1.2]{BBCK09} and \cite[Proposition 3.1]{CKK11}.
\begin{thm}{\rm (\cite[Theorem 1.2]{BBCK09} and \cite[Proposition 3.1]{CKK11})} \label{l:upp}
Under Assumption {\rm \ref{assum:jump}}, there are a properly exceptional set $\N\subset \R^d$,
a non-negative symmetric kernel $p(t,x,y)$ defined on
$(0,\infty)\times (\R^d\setminus \N) \times (\R^d\setminus \N)$ such that $P(t,x,\d y)=p(t,x,y)\,\d y$, and
$$p(t,x,y)\le c_0(t^{-d/\alpha_1}\vee t^{-d/2}),\quad t>0, \ x,y\in \R^d\setminus \N$$
holds with some constant $c_0>0$. Moreover, there is
an $\E$-nest $\{F_k:k\ge1\}$ of compact subsets of $\R^d$ so that $${\cal N}={\mathbb R}^d\setminus \bigcup_{k=1}^{\infty}F_k$$
and that for each fixed  $t>0$ and $y\in {\mathbb R}^d\setminus {\cal N}$, the map
$x\mapsto p(t,x,y)$ is continuous on each $F_k.$ \end{thm}

To obtain upper bounds of off-diagonal estimates for $p(t,x,y)$, we will use the following Davies' method, see \cite{CKS87}. Note that, the so-called {\it carr\'{e} du champ} associated with $(\E,\F)$ is given by
$$\Gamma(f,g)(x)=\int_{\R^d} (f(y)-f(x))(g(y)-g(x))J(x,y)\,\d y,\quad f,g\in \F.$$
We can extend $\Gamma (f,f)$ to any non-negative measurable function $f$, whenever it is pointwise well defined.

The following proposition immediately follows from Theorem \ref{l:upp} and  \cite[Corollary 3.28]{CKS87}.
\begin{prop}\label{P:da}
Suppose that Assumption {\rm \ref{assum:jump}} holds.
Then, there exists a constant $c_0>0$ such that 
for any $x,y\in \R^d\setminus \N$ and $t>0$,
$$p(t,x,y)\le c_0 (t^{-d/\alpha_1}\vee t^{-d/2}) \exp\left(E(2t, x,y)\right),$$
where
$$ E(t,x,y):=-\sup\{|\psi(x)-\psi(y)|-t\Lambda(\psi): \psi\in C_c^{\rm lip}(\R^d) \hbox{ with }\Lambda(\psi)<\infty\}
$$ and
$$\Lambda(\psi):=\|e^{-2\psi}\Gamma(e^\psi, e^\psi)\|_\infty.$$
\end{prop}

\ \

In the next section, we will consider the following two assumptions on the jumping kernel $J(x,y)$ for $x,y\in \R^d$ with $|x-y|\ge1$.
\begin{enumerate}
\item[{\bf(A)}] There are a constant $c>0$ and
an increasing function $\phi: [1,\infty)\rightarrow (1,\infty]$ such that
for all $x,y\in\R^d$ with $|x-y|\ge 1$,
\begin{equation}
\label{upp-2} J(x,y)\leq \frac{c}{|x-y|^{d+2}\phi(|x-y|)}.
\end{equation}
Moreover, the function
$$\Phi(s):= \left(\int_s^{\infty}\frac{{\rm d}r}{r\phi(r)}\right)^{-1},\quad s\ge 1$$
satisfies
\begin{itemize}
\item $\Phi(\infty)=\infty$;
\item the function $s\mapsto \log \Phi(s)/s$ is decreasing on $[1,\infty)$;
\item there is a constant $\gamma>0$ such that
\begin{equation}\label{sec00}\sup_{s\ge1}\frac{\Phi(s)}{\phi^\gamma(s)}<\infty.\end{equation}
\end{itemize}
\item[{\bf(B)}]  There is a constant $c>0$ such that
for all $x,y\in\R^d$ with $|x-y|\ge 1$,
\begin{equation}
\label{upp-1}J(x,y)\leq \frac{c}{|x-y|^{d+2}}.
\end{equation}
It also holds that
\begin{equation}\label{sec}\sup_{x\in \R^d} \int_{\{|x-y|\ge1\}} |x-y|^2J(x,y)\,\d y<\infty.\end{equation}

\end{enumerate}

Because $\phi$ is increasing on $[1,\infty)$, \eqref{upp-2} is stronger than \eqref{upp-1}.
Since the condition $\Phi(\infty)=\infty$ implies \eqref{sec},
{\bf (A)} is stronger than {\bf (B)}.
For instance,
$\phi(r)=(1+r)^{\theta}$, $\phi(r)=\log^{1+\theta} (e+r)$
and $\phi(r)=\log(e+r)\log^{1+\theta}\log(e^e+r)$ for any $\theta>0$ satisfy the conditions in {\bf(A)}.
On the other hand, under \eqref{A:jump-kernel} and \eqref{sec},
$$\sup_{x\in\R^d}\int_{\R^d} |x-y|^2 J(x,y)\,\d y<\infty.$$  In particular, there is a constant $c_1>0$ such that for any $K>0$,
\begin{equation}\label{big-jump}
\sup_{x\in\R^d}\int_{\{|x-y|>K\}} J(x,y)\,\d y\le \frac{c_1}{K^2}.
\end{equation}

\section{Heat kernel estimates}\label{section3}
Throughout this section, we always suppose that Assumption {\rm \ref{assum:jump}} holds.
We will derive upper and lower bound estimates of the heat kernel for large time respectively.

\subsection{Heat kernel upper bound}

\begin{prop}\label{thm-upper-bound0}
Under Assumption {\bf (B)},
there exist positive constants $t_0$ and $c$ such that for all $t\ge t_0$ and $x,y\in \R^d\backslash \N,$
$$p(t,x,y)\leq
\begin{cases}
\displaystyle \frac{c}{t^{d/2}},& t\geq |x-y|^2, \\
\displaystyle \frac{ct}{|x-y|^{d+2}}, & t\leq |x-y|^2.
\end{cases}$$
\end{prop}
\begin{proof}
We mainly follow the proof of \cite[Theorem 1.4]{BGK09},
but here we suppose that the time parameter $t$ is large.
By Theorem \ref{l:upp}, there are  constants $t_0, c_0>0$ such that for all $x,y\in \R^d\backslash \N$ and $t\ge t_0$,
$$p(t,x,y)\leq {c_0}{t^{-d/2}}.$$
Thus, we only need to verify the off-diagonal estimate for $p(t,x,y).$

We first introduce  truncated Dirichlet forms associated with $(\E,\F)$.
For $0<K<\infty$, define
$${\cal E}^{(K)}(u,v)=\iint_{\{0<|x-y|<K\}} (u(x)-u(y))(v(x)-v(y))\,J(x,y)\,\d x\,\d y, \quad u,v\in {\cal F}.$$
Then by (\ref{big-jump}),
\begin{equation*}
\begin{split}
\iint_{\{|x-y|\geq K\}} (u(x)-u(y))^2J(x,y)\,\d x\,\d y
&\leq 4\int_{\R^d}u(x)^2 \left(\int_{\{|x-y|\geq K\}}J(x,y)\,\d y\right)\,\d x\\
&\leq \frac{c_1}{K^2}\|u\|_2^2,
\end{split}
\end{equation*}
which yields that
\begin{equation}\label{e:ffee1} \begin{split}
{\cal E}(u,u)
=&{\cal E}^{(K)}(u,u)+\iint_{\{|x-y|\geq K\}} (u(x)-u(y))^2J(x,y)\,\d x\,\d y\\
\leq & {\cal E}^{(K)}(u,u)+\frac{c_1}{K^2}\|u\|_2^2.\end{split}
\end{equation} In particular, $({\cal E}^{(K)},{\cal F})$ is a regular Dirichlet form on $L^2(\R^d;\d x)$.

Let $P^{(K)}(t,x,\d y)$ be the transition probability associated with $({\cal E}^{(K)},{\cal F})$.
Then, by \eqref{e:ffee1} and the proof of \cite[Theorem 1.2]{BBCK09} (or \cite[Proposition 3.1]{CKK11}), there exist positive constants $c_2, c_3$ and $t_1$ such that for all $t\ge t_1$ and $x,y\in\R^d\backslash \N$,
$$P^{(K)}(t,x,\d y)=p^{(K)}(t,x,y)\,\d y$$ and
\begin{equation}\label{e:ont}
p^{(K)}(t,x,y)\le c_2 t^{-d/2}\exp\left(\frac{c_3t}{K^2}\right).
\end{equation}

Next, we will obtain the off-diagonal estimate for $p^{(K)}(t,x,y)$, by applying Proposition \ref{P:da} to $({\cal E}^{(K)},{\cal F})$. For fixed points $x_0, y_0\in \R^d$,
let $R=|x_0-y_0|$ and $K=R/\theta$ for some $\theta>0$,
which will be determined later.
For $\lambda>0$, we define the function $\psi\in C_c^{\rm lip}(\R^d)$ by
$$\psi(x)=[\lambda(R-|x-y_0|)]\vee 0.$$
Then, by the inequality $(e^r-1)^2\leq r^2e^{2|r|}$ for $r\in {\mathbb R}$ and the fact that $|\psi(x)-\psi(y)|\leq \lambda |x-y|$ for all $x,y\in\R^d$, we get
\begin{equation}\label{gamma-upper}
\begin{split}
\Gamma_K(\psi)(x):&= e^{-2\psi(x)} \Gamma^{(K)}(e^{\psi}, e^{\psi})(x)\\
&=\int_{\{0<|x-y|<K\}}\left(e^{\psi(y)-\psi(x)}-1\right)^2J(x,y)\,{\rm d}y\\
&\leq \int_{\{0<|x-y|<K\}}\left(\psi(x)-\psi(y)\right)^2 e^{2|\psi(x)-\psi(y)|}J(x,y)\,{\rm d}y\\
&\leq e^{2\lambda K}\lambda^2 \int_{\{0<|x-y|<K\}}|x-y|^2J(x,y)\,\,{\rm d}y \\
&\leq c_4\lambda ^2 e^{2\lambda K}\leq c_5\frac{e^{3\lambda K}}{K^2},
\end{split}
\end{equation} where in the third inequality we used \eqref{sec} and the last inequality follows from the fact that $r^2\le 2e^r$ for all $r\geq 0$. Hence,
$$\Lambda_K(\psi):=\| \Gamma_K(\psi)\|_\infty \le c_5\frac{e^{3\lambda K}}{K^2},$$
which implies that
\begin{equation}\label{upper-exp}
E^{(K)}(t,x_0,y_0)\le -|\psi(x_0)-\psi(y_0)|+\Lambda(\psi)t \leq c_5\frac{e^{3\lambda K}}{K^2}t-\lambda R.
\end{equation}

In what follows, we assume that  $t<K^2$.
In  (\ref{upper-exp}), if we take
$$\lambda=\frac{1}{3K}\log\left(\frac{K^2}{t}\right),$$
then
$$E^{(K)}(t,x_0,y_0)\leq -\frac{R}{3K}\log\left(\frac{K^2}{t}\right)+\frac{c_5}{K^2}\frac{K^2}{t}t
=c_5-\frac{\theta}{3}\log\left(\frac{K^2}{t}\right)$$
so that by \eqref{e:ont} and Proposition \ref{P:da},
\begin{align*}p^{(K)}(t,x_0,y_0)\leq & c_6t^{-d/2}\exp\left(\frac{c_3t}{K^2}+E^{(K)}(2t,x_0,y_0)\right)\\
\le & c_6 t^{-d/2}\exp\left(c_3+c_5-\frac{\theta}{3}\log\left(\frac{K^2}{2t}\right)\right)\\
=& c_7t^{-d/2}\left(\frac{2t}{K^2}\right)^{\theta/3}.\end{align*}
Hence by letting $\theta=3(d+2)/2$, we have
\begin{equation}\label{e:upper-trunc0}
p^{(K)}(t,x_0,y_0)\leq c_7t^{-d/2}\left(\frac{2t}{K^2}\right)^{(d+2)/2}
=\frac{c_8t}{K^{d+2}}=\frac{c_8\theta^{d+2}t}{|x_0- y_0|^{d+2}}.
\end{equation}

We finally obtain the off-diagonal upper bound of $p(t,x,y)$.
In fact,
by Meyer's construction (see e.g.\ \cite[Lemma 3.1(c)]{BGK09} or \cite[Lemma 3.7(b)]{BBCK09}),
\eqref{e:upper-trunc0} and \eqref{upp-1},
\begin{equation}\label{e:meyer}
p(t,x_0,y_0)\leq p^{(K)}(t,x_0,y_0)+t\sup_{|x-y|\geq K}J(x,y)\leq \frac{c_9t}{|x_0-y_0|^{d+2}}.
\end{equation}
Therefore, the proof is complete.
\end{proof}

\begin{thm}\label{thm-upper-bound}
Suppose that Assumption {\bf (A)} holds.
Then, for any $\kappa\ge1$,
there exist positive constants $\theta_0\in (0,1)$, $t_0\ge 1$ and $c_i \ (i=1,2)$ such that
for all $t\ge t_0$,
$$
p(t,x,y)
\leq
\begin{cases}
\displaystyle \frac{c_1}{t^{d/2}},& t\geq |x-y|^2,\\
\displaystyle \frac{c_1}{t^{d/2}}\exp\left(-\frac{c_2|x-y|^2}{t}\right),
& \displaystyle  \frac{\theta_0|x-y|^2}{\log \Phi(|x-y|)}\leq t\leq |x-y|^2,\\
\displaystyle U(t,|x-y|,\phi,\Phi, \kappa),
& \displaystyle t\leq \frac{\theta_0|x-y|^2}{\log \Phi(|x-y|)},
\end{cases}
$$ where $$U(t,|x-y|,\phi, \Phi,\kappa):=\frac{c_1}{t^{d/2}\Phi(|x-y|/\kappa)^{\kappa/8}} \wedge \frac{c_1t}{|x-y|^{d+2}}+\frac{c_1t}{|x-y|^{d+2}\phi(|x-y|/\kappa)}.$$
\end{thm}

\begin{proof}
We use the same notations as in those 
of Proposition \ref{thm-upper-bound0}.
By Theorem \ref{l:upp}, we only need to consider off-diagonal estimates, i.e.,
the case that $t\le |x-y|^2$.
We split the proof into two parts. Even though 
the proof below is based on the Davies method,
the argument is much more delicate than that of Proposition \ref{thm-upper-bound0}.

Let $K\geq 1$. For fixed points $x_0, y_0\in \R^d$ with $|x_0-y_0|\geq 1$,
let $R=|x_0-y_0|$. For $\lambda>0$, define the function $\psi\in C_c^{\rm lip}(\R^d)$ by
$$\psi(x)=[\lambda (R-|x-y_0|)]\vee 0.$$
Then 
by the same argument as in (\ref{gamma-upper}),
and by Assumption \ref{assum:jump} (ii) and  Assumption {\bf(A)},
\begin{equation}\label{e:poff}
\begin{split}
\Gamma_K(\psi)(x)&=\int_{\{0<|x-y|<K\}}\left(e^{\psi(y)-\psi(x)}-1\right)^2J(x,y)\,{\rm d}y\\
&\leq \lambda ^2\int_{\{0<|x-y|<K\}}|x-y|^2e^{2\lambda |x-y|}J(x,y)\,\,{\rm d}y \\
&=\lambda ^2\int_{\{0<|x-y|<1\}}|x-y|^2e^{2\lambda |x-y|}J(x,y)\,\,{\rm d}y\\
&\quad +\lambda ^2\int_{\{1\leq |x-y|<K\}}|x-y|^2e^{2\lambda |x-y|}J(x,y)\,\,{\rm d}y \\
&\leq \lambda ^2e^{2\lambda}
\sup_{x\in \R^d}\int_{\{0<|x-y|<1\}}|x-y|^2J(x,y)\,\,{\rm d}y\\
&\quad +c_1\lambda^2\int_{\{1\leq |x-y|<K\}}
\frac{e^{2\lambda |x-y|}}{|x-y|^d\phi(|x-y|)}\,\,{\rm d}y\\
&=:{\rm (I)}+{\rm (II)}.
\end{split}
\end{equation}

(1) We first derive the desired Gaussian upper bound.
For any $\theta>0$, let $\eta$ be a positive constant such that $\eta/\theta<1/4$.
Assume  that $K=R$ and $t\geq \theta K^2/\log \Phi(K)$. We set $\lambda=\eta K/t$.
Since $K\ge 1$ and the function $s\mapsto\log \Phi(s)/s $ is decreasing on $[1,\infty)$ by Assumption {\bf (A)},
$$e^{2\lambda}=e^{2\eta K/t}
\leq \exp\left(2\eta\frac{\log \Phi(K)}{\theta K}\right)
\leq e^{2\eta\log\Phi(1)/\theta}=\Phi(1)^{2\eta/\theta},$$
and so
$$
{\rm (I)}\leq c_2\Phi(1)^{2\eta/\theta}\lambda^2\le  c_2(1+\Phi(1))^{2\eta/\theta}\lambda^2
\le c_2(1+\Phi(1))^{1/2}\lambda^2 =: c_3 \lambda^2.
$$
If $1\leq r\leq K$, then, also due to the decreasing property of the function $s\mapsto\log \Phi(s)/s$,
$$e^{2\lambda r}=e^{2\eta Kr/t}
\leq \exp\left(2\eta r\frac{\log \Phi(K)}{\theta K}\right)
\leq \exp\left(2\eta r\frac{\log \Phi(r)}{\theta r}\right)=\Phi(r)^{2\eta/\theta},$$
which implies that
\begin{align*}
{\rm (II)}
&\leq  c_1\lambda^2\int_{\{|x-y|\geq 1\}}
\frac{\Phi(|x-y|)^{2\eta/\theta}}{|x-y|^d \phi(|x-y|)}\,\,{\rm d}y
=c_4 \lambda^2 \int_1^{\infty}\frac{\Phi(r)^{2\eta/\theta}}{r\phi(r)}\,\d r\\
&= c_4\lambda^2 \int_1^\infty \frac{1}{r\phi(r)}  \left(\int_r^\infty \frac{1}{s\phi(s)}\,\d s\right)^{-2\eta/\theta}\,\d r
= \frac{c_4\lambda^2}{1-(2\eta/\theta)}\left(\int_1^\infty \frac{1}{s\phi(s)}\,\d s\right)^{1-(2\eta/\theta)} \\
&\le {2c_4\lambda^2} \left(1+\int_1^\infty \frac{1}{s\phi(s)}\,\d s\right) =:c_5\lambda^2.
\end{align*}
Hence by \eqref{e:poff},
$$\Lambda_{K}(\psi)=\|\Gamma_K(\psi)\|_{\infty}
\leq (c_3+c_5)\lambda^2=:C_*\lambda^2.$$
In particular, we have
$$E^{(K)}(t,x_0,y_0)\le \Lambda_{K}(\psi)t-| \psi(x_0)-\psi(y_0)|\le
C_{*}\lambda^2 t -\lambda R=-\eta(1-\eta C_{*})\frac{K^2}{t}.$$
This along with Proposition \ref{P:da} yields that
there is a constant $c_6>0$ such that for all  $t\geq \theta K^2/\log \phi(K)$,
\begin{equation}\label{e:upper-trunc}
p^{(K)}(t,x_0,y_0)\leq c_6t^{-d/2}\exp\left\{\frac{c_0t}{K^2}-\frac{\eta(1-\eta C_{*})}{2}\frac{K^2}{t}\right\}.
\end{equation}
We note that the constants $c_6$ and $C_*$ above are independent of $\eta$ and $\theta$.

In what follows, we assume that
$$
\frac{\theta K^2}{\log \Phi(K)}\leq t\leq K^2.$$
Since by \eqref{e:upper-trunc},
$$
p^{(K)}(t,x_0,y_0)
\leq c_7t^{-d/2}\exp\left\{-\frac{\eta(1-\eta C_{*})}{2}\frac{K^2}{t}\right\},
$$
we have by the first inequality in \eqref{e:meyer} and 
\eqref{upp-2},
\begin{equation}\label{comparison}
\begin{split}
p(t,x_0,y_0)
&\leq p^{(K)}(t,x_0,y_0)+t\sup_{|x-y|\geq K} J(x,y)\\
&\leq c_7t^{-d/2}\exp\left\{-\frac{\eta(1-\eta C_{*})}{2}\frac{K^2}{t}\right\}
+\frac{c_8t}{K^{d+2}\phi(K)}.
\end{split}
\end{equation}

Let $\eta_*$ be a positive constant such that
$$\frac{\eta_*((1-\eta_*C_*)\vee 4)}{2\theta}\in \Big(0,1\wedge \frac{1}{\gamma}\Big ),$$
where $\gamma$ is the constant in Assumption {\bf (A)}.
Then by \eqref{sec00}, there is a constant $c_9>0$ such that
$$
\exp\left\{-\frac{\eta_*(1-\eta_* C_{*})}{2}\frac{K^2}{t}\right\}
\geq\exp\left\{-\frac{\eta_*(1-\eta_* C_{*})}{2}\frac{\log\Phi(K)}{\theta}\right\}
=\frac{1}{\Phi(K)^{\eta_*(1-\eta_* C_*)/(2\theta)}}\geq \frac{c_9}{\phi(K)}.
$$
By noting that
$$
\frac{1}{t^{d/2}}=\frac{t}{t^{(d+2)/2}}
\geq t\left(\frac{1}{K^2}\right)^{(d+2)/2}
=\frac{t}{K^{d+2}},
$$
we get
\begin{equation*}
\begin{split}
\frac{t}{K^{d+2}\phi(K)}
\leq c_9^{-1}t^{-d/2}\exp\left\{-\frac{\eta_*(1-\eta_* C_{*})}{2}\frac{K^2}{t}\right\}.
\end{split}
\end{equation*}
Hence if we take $\eta=\eta_*$ in \eqref{comparison}, then
\begin{equation*}
\begin{split}
p(t,x_0,y_0)
\leq&
c_7t^{-d/2}\exp\left\{-\frac{\eta_*(1-\eta_* C_{*})}{2}\frac{K^2}{t}\right\}\\
&+c_{10} t^{-d/2}
\exp\left\{-\frac{\eta_*(1-\eta_* C_{*})}{2}\frac{K^2}{t}\right\}\\
=&:c_*t^{-d/2}\exp\left\{-\frac{\eta_*(1-\eta_* C_{*})}{2}\frac{|x_0-y_0|^2}{t}\right\}.
\end{split}
\end{equation*}
Namely, for each fixed $\theta>0$,
we get the desired Gaussian bound for any $t>0$ and $x,y\in \R^d$ such that
$$\frac{\theta|x-y|^2}{\log \Phi(|x-y|)}\leq t\leq |x-y|^2.$$

(2) Let $\kappa\ge1$. Here we let  $K=R/\kappa$.
Since we can choose $t_0$ in the statement large enough,
we may and do assume
that $|x_0-y_0|$ is  large enough such that $|x_0-y_0|\ge \kappa$, and so $K\ge 1$.
Below we assume that
$$t\le \frac{\theta_0 R^2}{\log \Phi(R)}$$
for some $\theta_0>0$ small enough, which will be determined later.

Let
$$\lambda=\frac{\log \Phi(K)}{4K}.$$
Since the function $s\mapsto\log \Phi(s)/s $ on $[1,\infty)$ is decreasing by Assumption {\bf (A)},
$$e^{2\lambda r}= \exp\left(r\frac{ \log \Phi(K)}{2K}\right)\le \exp\left(r\frac{\log \Phi(r)}{2r}\right)
=\Phi(r)^{1/2},\quad 1\le r\le K.$$
Hence by \eqref{e:poff},
$$\Lambda_{K}(\psi)\le c_0\lambda^2,$$
where $c_0>0$ is independent of $\theta_0, \kappa$ and $\lambda.$
In particular, by choosing $\theta_0\in(0,1)$ so small that $c_0 \kappa\theta_0\le 2,$ we have
\begin{align*}E^{(K)}(t,x_0,y_0)&\le \Lambda_{K}(\psi)t-|\psi(x_0)-\psi(y_0)|\\
&\le
c_0\lambda^2 t -\lambda R\\
&\le \frac{c_0}{16} \left(\frac{ \log \Phi(K)}{K}\right)^2\frac{\theta_0R^2}{\log \Phi(R)}- \frac{\log \Phi(K)}{4K} R\\
&=\frac{\kappa}{4} \log \Phi(K) \left(-1+\frac{c_0 \kappa\theta_0 }{4}\frac{\log \Phi(K)}{\log \Phi(\kappa K)} \right)\\
&\le -\frac{\kappa}{8} \log \Phi(K),\end{align*}
where we used $\kappa\ge1$ and the increasing property of the function $\Phi(r)$ in the last inequality.
We then have by Proposition \ref{P:da},
$$
p^{(K)}(t,x_0,y_0)
\leq c_1t^{-d/2} \frac{1}{\Phi(K)^{\kappa/8}},
$$ which yields that by the same way as in \eqref{comparison},
$$ p(t,x_0,y_0)\le c_1t^{-d/2} \frac{1}{\Phi(|x_0-y_0|/\kappa)^{\kappa/8}}+\frac{c_2t}{|x_0-y_0|^{d+2}\phi(|x_0-y_0|/\kappa)}.$$

Noting that Assumption {\bf (B)} is weaker than Assumption {\bf(A)},
we know from Proposition \ref{thm-upper-bound0} that for
any $x_0,y_0\in \R^d\setminus\N$ and $t\ge t_0$ with $t\le |x_0-y_0|^2$,
$$ p(t,x_0,y_0)\le \frac{c_3t}{|x_0-y_0|^{d+2}}.$$
Since $\phi$ is an increasing function on $[1,\infty)$ and $|x_0-y_0|\geq \kappa$,
we have $\phi(|x_0-y_0|/\kappa)\geq \phi(1)$ so that
$$\frac{t}{|x_0-y_0|^{d+2}\phi(|x_0-y_0|/\kappa)}\le  \frac{t}{\phi(1)|x_0-y_0|^{d+2}}.
$$
Therefore, we finally obtain
$$p(t,x_0,y_0)\le \frac{c_4}{t^{d/2}\Phi(|x-y|/\kappa)^{\kappa/8}} \wedge \frac{c_4t}{|x-y|^{d+2}}
+\frac{c_4t}{|x-y|^{d+2}\phi(|x-y|/\kappa)}.$$

\ \

Combining the conclusions in (1) and (2) above, we get the desired assertion.
\end{proof}

\begin{rem}\label{rem-upper-bound}\rm (i) According to Theorem \ref{thm-upper-bound}, we can obtain \cite[Theorem 3.3]{CKK11}
when $\phi(r)=\exp({cr^\beta})$ for some constants $c>0$ and $\beta\in(0,1]$. By \cite[(1.14) in
Theorem 1.2]{CKK11}, we know that upper bound estimates in Theorem \ref{thm-upper-bound} are sharp
up to constants in this case.

(ii) By part (1) of the argument for Theorem \ref{thm-upper-bound}, we indeed prove that for any $\theta>0$, there are constants $c_i=c_i(\theta)>0$ $(i=1,2)$ such that for all $t\ge t_0$ and $x,y\in \R^d$ with
$$\frac{\theta|x-y|^2}{\log \Phi(|x-y|)}\leq t\leq |x-y|^2,$$ it holds that
$$p(t,x,y)\le \frac{c_1}{t^{d/2}}\exp\left(-\frac{c_2|x-y|^2}{t}\right).$$
\end{rem}

As a consequence of Theorem \ref{thm-upper-bound},
we have the following statement about upper bound estimates of  the heat kernel for a new class of symmetric jump processes.
\begin{cor}\label{c:two}
Assume that there are positive constants $\varepsilon, c_0$ such that for all $x,y\in \R^d$ with $|x-y|\ge 1$,
$$J(x,y)\le \frac{c_0}{|x-y|^{d+2+\varepsilon}}.$$ Then, there exist positive constants $t_0\ge 1$, $\theta_0>0$ and $c_i \ (i=1,2)$ such that for all $t\ge t_0$,
$$p(t,x,y)\leq
\begin{cases}
\displaystyle \frac{c_1}{t^{d/2}},& t\geq |x-y|^2,\\
\displaystyle \frac{c_1}{t^{d/2}}\exp\left(-\frac{c_2|x-y|^2}{t}\right),
& \displaystyle  \frac{\theta_0|x-y|^2}{\log(1+|x-y|)}\leq t\leq |x-y|^2,\\
\displaystyle \frac{c_1t}{|x-y|^{d+2+\varepsilon}}, & \displaystyle t\leq \frac{\theta_0|x-y|^2}{\log(1+|x-y|)}.
\end{cases}$$\end{cor}
\begin{proof}
In this case, $\phi(r)=r^\varepsilon$ and $\Phi(r)=\varepsilon r^{\varepsilon}$.
By taking $\kappa\ge 1$ so large enough that
$\varepsilon \kappa/8\ge d+2+\varepsilon$ in Theorem \ref{thm-upper-bound},
we obtain the desired assertion. \end{proof}

To study rate functions of the process $X$ corresponding to the test function $\phi(r)=\log^{1+\varepsilon}r$, we also need the following.
\begin{prop}\label{thm-upper-bound3}
Suppose that Assumption {\bf (A)} is satisfied.
Then for any $\delta\in(0,1)$, there exist positive constants $t_0$, $\theta_0$ and $c_1, c_2$ such that
$$p(t,x,y)\leq \frac{c_1t}{|x-y|^{d+2}\log^{(d+2)\delta/2}\log(e+\Phi(c_2|x-y|))}$$
for all $t\ge t_0$ and $x,y\in \R^d\backslash \N$ with
$$t_0\le t\le \frac{ \theta_0|x-y|^2}{\log\Phi(|x-y|)}.$$
\end{prop}
\begin{proof}
For fixed points $x_0, y_0\in {\mathbb R}^d$ and $\theta>0$,
we let $R=|x_0-y_0|$ and $K=R/\theta$.
Since $t_0$ can be large enough,
we may and do assume that $R$ is large enough.
We use the approach of Proposition \ref{thm-upper-bound0}
and start from the estimate \eqref{upper-exp}.
Taking
$$
\lambda= \frac{1}{3K} \log\left (\frac{K^2 \log^\delta \log \Phi(K)}{t}\right),$$ we have
$$E^{(K)}(t,x_0,y_0)\le - \frac{\theta}{3} \log \left(\frac{K^2 \log^\delta \log \Phi(K)}{t}\right)
+ c_* \log^\delta\log \Phi(K),$$
where $c_*$ is the constant $c_5$ in \eqref{upper-exp}.
If
$$t\le \frac{ c_0 K^2}{\log\Phi(K)}$$
for some $c_0>0$, then for $K\ge 1$ large enough,
$$ \frac{\theta}{6}\log \left(\frac{K^2 \log^\delta \log \Phi(K)}{t}\right)
\ge \frac{\theta}{6}\log \left(\frac{\log \Phi(K) \log^\delta \log \Phi(K)}{c_0}\right)
\ge  c_* \log^\delta\log \Phi(K),$$
due to the fact that $\delta\in(0,1)$.
Hence, for $K\ge 1$ large enough, we have
$$E^{(K)}(t,x_0,y_0)\le - \frac{\theta}{6} \log \left(\frac{K^2 \log^\delta \log \Phi(K)}{t}\right),$$
which along with Proposition \ref{P:da} yields that
\begin{equation*}
\begin{split}
p^{(K)}(t,x_0,y_0)
&\le c_1 t^{-d/2} \exp\left(-\frac{\theta}{6} \log \left(\frac{K^2 \log^\delta \log \Phi(K)}{2t}\right)\right)\\
&= c_1 t^{-d/2}\left(\frac{2t}{K^2 \log^\delta \log \Phi(K)}\right)^{\theta/6}.
\end{split}
\end{equation*}
Setting $\theta= 3(d+2)$, we get
$$p^{(K)}(t,x_0,y_0)\le c_2 \frac{t}{K^{d+2} \log^{(d+2)\delta/2}\log \Phi(K)}.$$
This along with the first inequality in \eqref{e:meyer}, Assumption {\bf (A)} and the fact that $|x_0-y_0|=\theta K$ gives us that
\begin{align*}
p(t,x_0,y_0)
&\leq p^{(K)}(t,x_0,y_0)+t\sup_{|x-y|\geq K}J(x,y)\\
&\leq c_3 \frac{t}{|x_0-y_0|^{d+2} \log^{(d+2)\delta/2}\log \Phi(c_4|x_0-y_0|)}.
\end{align*}
The proof is complete. \end{proof}

\subsection{Heat kernel lower bound}
In this subsection, we establish the following lower bound estimates for the heat kernel.
\begin{thm}\label{T:lower}
Under Assumption {\bf(B)}, there exist positive constants $t_0$ and $c_i$  $(i=1,2,3)$ such that
for all $t\ge t_0$ and $x,y\in \R^d\backslash \N$,
$$p(t,x,y)\geq
\begin{cases}
c_1t^{-d/2} & |x-y|^2\leq t, \\
\displaystyle c_1t^{-d/2}\exp\left(-\frac{c_2|x-y|^2}{t}\right) &
c_3|x-y|\leq t\leq |x-y|^2.
\end{cases}$$
\end{thm}

We first explain the main idea of the proof of Theorem \ref{T:lower}.
Following the approach of \cite{BBCK09}, we introduce a class of modifications for the jumping kernel $J(x,y)$.
Let $\kappa_2$ be the constant in \eqref{A:jump-kernel}.
For $\delta\in (0,1)$, define
\begin{equation}\label{eq-j-delta}
J^{(\delta)}(x,y):=J(x,y){\bf 1}_{\{|x-y|\geq \delta\}}+\frac{\kappa_2}{|x-y|^{d+\alpha_2}}{\bf 1}_{\{0<|x-y|<\delta\}}
\end{equation}
and
$${\cal D}^{\delta}:=\left\{u\in L^2({\mathbb R}^d;\d x) \bigg|
\iint_{x\ne y}(u(x)-u(y))^2J^{(\delta)}(x,y)\,{\rm d}x\,{\rm d}y<\infty\right\}.$$
Then by Assumption \ref{assum:jump}, we have for any $\delta\in(0,1)$
\begin{equation*}
\begin{split}
\iint_{\{|x-y|\geq \delta \}}(u(x)-u(y))^2J(x,y)\,{\rm d}x\,{\rm d}y
&\leq 4\int u(x)^2\left(\int_{\{|x-y|\geq \delta\}}J(x,y)\,{\rm d}y\right)\,{\rm d}x\\
&\leq c_1(\delta)\int u(x)^2\,{\rm d}x
\end{split}
\end{equation*}
and so
\begin{equation}\label{eq-compare}\begin{split}
\iint_{x\ne y}(u(x)-u(y))^2J^{(\delta)}(x,y)\,{\rm d}x\,{\rm d}y&
+\|u\|_{L^2({\mathbb R}^d;\d x)}^2\\
&\asymp
\iint_{x\ne y}\frac{(u(x)-u(y))^2}{|x-y|^{d+\alpha_2}}\,{\rm d}x\,{\rm d}y
+\|u\|^2_{L^2({\mathbb R}^d;\d x)}.\end{split}
\end{equation}
Therefore, for all $\delta\in (0,1)$,
$${\cal D}^{\delta}
=\left\{u\in L^2({\mathbb R}^d;\d x) \bigg|
\iint_{x\ne y}\frac{(u(x)-u(y))^2}{|x-y|^{d+\alpha_2}}\,{\rm d}x\,{\rm d}y<\infty \right\};$$
that is, ${\cal D}^{\delta}$ is independent of $\delta\in (0,1)$.

Let $({\cal E}^{\delta},{\cal D}^{\delta})$ be a bilinear form on $L^2({\mathbb R}^d;\d x)$ given by
$${\cal E}^{\delta}(u,v)=\iint_{{\mathbb R}^d\times {\mathbb R}^d}(u(x)-u(y))(v(x)-v(y)) J^{(\delta)}(x,y)\,{\rm d}x\,{\rm d}y,
\quad \text{$u,v\in {\cal D}^{\delta}$},$$
and let ${\cal F}^{\delta}$ be the closure of $C_c^{\rm lip}(\R^d)$
with respect to the norm $\|f\|_{\E_1^{\delta}}:=\sqrt{\E^{\delta}(f,f)+\|f\|_2^2}$ in ${\cal D}^{\delta}$.
Then, $({\cal E}^{\delta},{\cal F}^{\delta})$ is a regular Dirichlet form on $L^2({\mathbb R}^d;\d x)$.
Moreover, according to (\ref{eq-compare}) and the argument of \cite[Lemma 2.5]{BBCK09},
we have ${\cal F}^{\delta}={\cal D}^{\delta}$.

Associated with the regular Dirichlet form $({\cal E}^{\delta}, {\cal F}^{\delta})$ is
a symmetric Hunt process $Y^{\delta}=(\{Y_t^{\delta}\}_{t\geq 0}, \{\Pp^x\}_{x\in \R^d\setminus\N})$
with state space $\R^d\backslash \N_\delta$,
where $\N_\delta\subset \R^d$ is a properly exceptional set for $({\cal E}^{\delta}, {\cal F}^{\delta})$.
By \cite[Main result]{MU11} the process $Y^{\delta}$ is conservative.
We also see from
Theorem \ref{l:upp} that there exists
a non-negative kernel $q^{\delta}(t,x,y)$
on $(0,\infty)\times ({\mathbb R}^d\setminus {\cal N}_{\delta})\times ({\mathbb R}^d\setminus {\cal N}_{\delta})$
such that for any non-negative function $f$ on ${\mathbb R}^d$,
$$\Ee^x f(Y_t^{\delta})=\int_{{\mathbb R}^d}q^{\delta}(t,x,y)f(y)\,{\rm d}y,
\quad \text{$t>0$ and $x\in {\mathbb R}^d\setminus {\cal N}_{\delta}$}$$
and there is a constant
$c_2>0$ such that
\begin{equation}\label{e:upper1}
q^{\delta}(t,x,y)\leq
c_2(t^{-d/2}\vee t^{-d/\alpha_1}), \quad
\text{$t>0$ and $x,y\in {\mathbb R}^d\setminus {\cal N}_{\delta}$}.
\end{equation}
Moreover, there exists an ${\cal E}^{\delta}$-nest $\{F_k^{\delta}\}_{k\geq 1}$ of compact sets such that
$${\cal N}_{\delta}={\mathbb R}^d\setminus \bigcup_{k=1}^{\infty}F_k^{\delta}$$
and for each fixed  $t>0$ and $y\in {\mathbb R}^d\setminus {\cal N}_{\delta}$, the map
$x\mapsto q^{\delta}(t,x,y)$ is continuous on each $F_k^{\delta}$.
Here we should note that the constant $c_2$ in \eqref{e:upper1} can be chosen to be independent of $\delta\in (0,1)$.
Indeed, by the definition of $J^{(\delta)}(x,y)$,
$$J^{(\delta)}(x,y)\geq
\frac{\kappa_1}{|x-y|^{d+\alpha_1}}{\bf 1}_{\{|x-y|<1\}}+J(x,y){\bf 1}_{\{|x-y|\geq 1\}}=:J_l(x,y)$$
for any $\delta\in(0,1)$ and $x,y\in \R^d$.
Then by following the argument of \cite[Theorem 1.2]{BBCK09} and \cite[Proposition 3.1]{CKK11},
we see that $c_2$ can be determined by $J_l(x,y)$, which is independent of $\delta$.

Actually, under Assumption {\bf(B)},
we can also get the following near-diagonal lower bound of $q^{\delta}(t,x,y)$,
which is the key to Theorem \ref{T:lower}.

\begin{prop}\label{thm-lower-1}
Under Assumption {\bf(B)}, there exist constants $t_0>0$ and $c_0=c_0(t_0)>0$,
which are independent of $\delta\in (0,1)$, such that
for any $t\geq t_0$ and $x,y\in {\mathbb R}^d\setminus {\cal N}_{\delta}$ with $|x-y|^2\leq t$,
$$q^{\delta}(t,x,y)\geq c_0t^{-d/2}.$$
\end{prop}

We will prove Proposition \ref{thm-lower-1} later,
and present the proof of Theorem \ref{T:lower} first.

\begin{proof}[Proof of Theorem $\ref{T:lower}$]
(1) We first claim that there exist an ${\cal E}$-properly exceptional set ${\cal N}$ and
constants $t_0,c_0>0$ such that
for any $t\geq t_0$ and $x,y\in {\mathbb R}^d\setminus {\cal N}$ with $|x-y|^2\leq t$,
$$p(t,x,y)\geq c_0t^{-d/2}.$$
Indeed, let $\{\delta_n\}_{n=1}^{\infty}$ be a decreasing sequence in $(0,1)$
such that $\delta_n \rightarrow 0$ as $n\rightarrow\infty$.
Then, by \cite[p.1969, Theorem 2.3]{BBCK09},
$({\cal E}^{\delta_n},{\cal F}^{\delta_n})$ converges to $({\cal E}, {\cal F})$
in the sense of Mosco as $n\rightarrow \infty$.
Since $J^{(\delta)}(x,y)\geq J(x,y)$
by definition, we have
${\cal F}^{\delta}\subset {\cal F}$ and
$${\cal E}^{\delta}(u,u)\geq {\cal E}{(u,u)} \quad \text{for any $u\in {\cal F}^{\delta}$}.$$
Therefore, any ${\cal E}^\delta$-exceptional set can be regarded as an ${\cal E}$-exceptional set.
Namely, we can choose an ${\cal E}$-exceptional set ${\cal N}$
so that $\bigcup_{n=1}^{\infty}{\cal N}_{\delta_n}\subset {\cal N}$.
On account of this, the desired assertion follows from Proposition \ref{thm-lower-1}
and \cite[p.1990--1991, Proof of Theorem 1.3]{BBCK09}.

(2) Next, we prove Theorem $\ref{T:lower}$ by following the argument of \cite[Theorem 3.6]{CKK08}.
Note that if $t\geq t_0$ and $|x-y|^2\leq t$, then our assertion follows from (1).
In what follows, we assume that $\sqrt{t_0}|x-y|\leq t\leq |x-y|^2$.

Let $l$ be the maximum of positive integers such that
$$\frac{t}{l}\leq \left(\frac{|x-y|}{l}\right)^2.$$
Since
\begin{equation}\label{eq-l}
\frac{|x-y|^2}{t}-1\leq l\leq \frac{|x-y|^2}{t},
\end{equation}
we have
\begin{equation}\label{eq-space-time}
\frac{1}{2}\left(\frac{|x-y|}{l}\right)^2\leq \frac{t}{l}\leq \left(\frac{|x-y|}{l}\right)^2
\end{equation}
and
\begin{equation}\label{eq-lower-time}
\frac{t}{l}\geq \frac{t^2}{|x-y|^2}\geq t_0.
\end{equation}

Let $\{x_i\}_{0\le i\le 6l}$ be a sequence on the line segment
joining $x_0=x$ and $x_{6l}=y$ such that
\begin{equation}\label{eq-sequence}
|x_k-x_{k-1}|=\frac{|x-y|}{6l} \quad \text{for any $k=1,\dots, 6l$}.
\end{equation}
Take a sequence $\{y_i\}_{0\le i\le 6l}$ such that
$y_0=x$, $y_{6l}=y$ and $y_k\in B(x_k, (6l)^{-1}|x-y|)$ for all $1\leq k\leq 6l-1.$  Then,  \eqref{eq-sequence} and \eqref{eq-space-time} imply that for any $1\leq k\leq 6l$,
$$
|y_k-y_{k-1}|
\leq |y_k-x_k|+|x_k-x_{k-1}|+|x_{k-1}-y_{k-1}|
\leq 3\cdot \frac{|x-y|}{6l}
=\frac{|x-y|}{2l}
\leq \sqrt{\frac{t}{l}}.
$$
Hence by  \eqref{eq-lower-time} and (1),
there exists a constant
$C=C(t_0)\in (0,1)$ such that
$$p\left(\frac{t}{l},y_{k-1},y_k\right)\geq C\left(\frac{t}{l}\right)^{-d/2}, \quad 1\leq k\leq 6l.$$
This, together with the Markov property of $p(t,x,y)$, implies that
\begin{equation*}
\begin{split}
p(t,x,y)
&=\int_{{\mathbb R}^d}\cdots
\int_{{\mathbb R}^d}p(t/l,x,y_1)\cdots p(t/l,y_{6l-1},y)\,{\rm d}y_1\cdots\,{\rm d}y_{6l-1}\\
&\geq \int_{B(x_1,(6l)^{-1}|x-y|)}\cdots
\int_{B(x_{6l-1},(6l)^{-1}|x-y|)}p(t/l,x,y_1)\cdots p(t/l,y_{6l-1},y)\,{\rm d}y_1\cdots\,{\rm d}y_{6l-1}\\
&\geq
C\left(\frac{t}{l}\right)^{-d/2}
\prod_{k=1}^{6l-1}\left\{C\left(\frac{t}{l}\right)^{-d/2}
|B(x_k,(6l)^{-1}|x-y|)|\right\}\\
&\geq c_1\left(\frac{t}{l}\right)^{-d/2}C^{6l},
\end{split}
\end{equation*}
where in the second inequality $|\cdot|$ denotes the $d$-dimensional Lebesgue measure, and the last inequality follows from \eqref{eq-space-time}.
Note that, by \eqref{eq-l}, we have
$$
C^{6l}\geq e^{-c_2l}
\geq \exp\left(-c_2\frac{|x-y|^2}{t}\right),$$ which, along with the estimate above, yields
the desired assertion.
\end{proof}

\ \

The remainder of this subsection is devoted to the proof of Proposition \ref{thm-lower-1}.
For this, we need Lemmas \ref{lem-g'} and  \ref{lem-dist} below.
These two lemmas are concerned with a class of scaled processes for the subprocess of $Y^{\delta}$ on a ball.

We begin with some results which are due to \cite{BBCK09,CKK08,CK10,Fo09}.
Let $B(x,r)$ be an open ball with radius $r>0$ centered at $x\in {\mathbb R^d}$, and $B_r=B(0,r)$.
Denote by $Y^{\delta,B_r}$ the subprocess of $Y^{\delta}$ on $B_r$.
Let  $q^{\delta,B_r}(t,x,y)$ and $({\cal E}^{\delta,B_r}, {\cal F}^{\delta,B_r})$ be
the heat kernel (also called Dirichlet heat kernel in the literature)
and the regular Dirichlet form associated with $Y^{\delta,B_r}$, respectively.

For a fixed $r>0$, define $$Y_t^{\delta,(r)}:=r^{-1}Y_{r^2 t}^{\delta}.$$
Then $Y^{\delta,(r)}=\left(\left\{Y_t^{\delta, (r)}\right\}_{t\geq 0}, \{\Pp^x\}_{x\in {\mathbb R}^d\backslash \N_\delta}\right)$
is a symmetric Hunt process on ${\mathbb R}^d\backslash \N_\delta$ such that
the associated  Dirichlet form $({\cal E}^{\delta,(r)}, {\cal F}^{\delta,(r)})$ on $L^2({\mathbb R}^d;\d x)$
is given by
$${\cal E}^{\delta,(r)}(u,v)
=\iint_{{\mathbb R}^d\times {\mathbb R}^d}(u(x)-u(y))(v(x)-v(y))r^{d+2}
J^{(\delta)}(rx,ry)\,{\rm d}x\,{\rm d}y$$ and
$${\cal F}^{\delta,(r)}
=\left\{u\in L^2({\mathbb R}^d;\d x) \Big|
\iint_{{\mathbb R}^d\times {\mathbb R}^d}\frac{(u(x)-u(y))^2}{|x-y|^{d+\alpha_2}}
\,{\rm d}x\,{\rm d}y<\infty\right\}.$$
Moreover, the associated heat kernel $q_r^{\delta}(t,x,y)$ satisfies
\begin{equation}\label{eq-heat-scale}
q_r^{\delta}(t,x,y)=r^d q^{\delta}(r^2t,rx,ry).
\end{equation}
Let $Y^{\delta,(r), B_1}$ be the subprocess of $Y^{\delta,(r)}$ on $B_1$.
Then the associated Dirichlet heat kernel $q_r^{\delta,B_1}(t,x,y)$ is given by
$$q_r^{\delta, B_1}(t,x,y)=r^dq^{\delta, B_r}(r^2t,rx,ry),\quad
\text{$t>0$ and $x,y\in B_1\setminus{\cal N}_{\delta}$}.$$
We denote by $({\cal E}^{\delta,(r),B_1}, {\cal F}^{\delta,(r), B_1})$
the associated regular Dirichlet form on $L^2(B_1;{\rm d}x)$.

In the following, let
$$\Phi(x)=C_{\Phi}(1-|x|^2)^{\frac{12}{2-\alpha_2}}{\bf1}_{B_1}(x),\quad x\in \R^d$$
for some constant $C_{\Phi}>0$ so that $\int_{B_1}\Phi(x)\,\d x=1$.
For each fixed $x_1\in B_1\setminus {\cal N}$, $r\geq 1$ and $\varepsilon\in (0,1)$, define
$$u_r(t,x):=q_r^{\delta,B_1}(t,x,x_1), \quad u_r^{\varepsilon}(t,x):=u_r(t,x)+\varepsilon$$
and
$$H_{\varepsilon}(t):=\int_{B_1}\Phi(y)\log u_r^{\varepsilon}(t,y)\,{\rm d}y.$$

\begin{prop}\label{p:domain}
Under Assumption {\bf(B)}, the next two assertions hold.
\begin{enumerate}
\item For each $t>0$, the function $\Phi(\cdot)/u_r^{\varepsilon}(t,\cdot)$
belongs to $\F^{\delta,(r),B_1}$.
\item The function $H_{\varepsilon}(t)$ is differentiable on $(0,\infty)$ and for each $t>0$,
\begin{equation}\label{e:deriv}
H_{\varepsilon}'(t)=
-{\cal E}^{\delta,(r),B_1}\left(u_r(t,\cdot),\frac{\Phi(\cdot)}{u_r^{\varepsilon}(t,\cdot)}\right).
\end{equation}
\end{enumerate}
\end{prop}

\begin{proof}
(i) \ For any $x,y\in B_1$,
$$\frac{\Phi(x)}{u_r^{\varepsilon}(t,x)}\leq \frac{1}{\varepsilon}\Phi(x)$$
and
\begin{equation*}
\begin{split}
\left|\frac{\Phi(x)}{u_r^{\varepsilon}(t,x)}-\frac{\Phi(y)}{u_r^{\varepsilon}(t,y)}\right|
&\leq \frac{1}{u_r^{\varepsilon}(t,x)}\left|\Phi(x)-\Phi(y)\right|
+\Phi(y)\left|\frac{1}{u_r^{\varepsilon}(t,x)}-\frac{1}{u_r^{\varepsilon}(t,y)}\right|\\
&=\frac{1}{u_r^{\varepsilon}(t,x)}|\Phi(x)-\Phi(y)|
+\frac{\Phi(y)}{u_r^{\varepsilon}(t,x)u_r^{\varepsilon}(t,y)}|u_r^{\varepsilon}(t,x)-u_r^{\varepsilon}(t,y)|\\
&\leq \frac{1}{\varepsilon}|\Phi(x)-\Phi(y)|+\frac{C_{\Phi}}{\varepsilon^2}|u_r(t,x)-u_r(t,y)|.
\end{split}
\end{equation*}
Then our assertion follows by  the strong version of the normal contraction property
(e.g., see the proof of \cite[Theorem 1.4.2 (ii)]{FOT11}).

\noindent
(ii) \ By (i), the right hand side of \eqref{e:deriv} is finite for any $t>0$.
Then our assertion follows by the same way as in \cite[Lemmas 4.1 and 4.7]{BBCK09} and \cite[Proposition 3.7]{Fo09}.
\end{proof}

\begin{lem}\label{lem-g'}
Under Assumption {\bf(B)},
there exist positive constants $c_1$ and $c_2$
such that for any $\varepsilon\in (0,1)$, $\delta\in(0,1)$, $x_1\in B_1\setminus{\cal N}_{\delta}$,
$t>0$ and $r\ge1$,
\begin{equation}\label{eq-g'}
\begin{split}
H_{\varepsilon}'(t)\geq -c_1+c_2\int_{B_1}\left(\log u_r^{\varepsilon}(t,y)-H_{\varepsilon}(t)\right)^2\,\Phi(y)\,{\rm d}y.
\end{split}
\end{equation}
\end{lem}

\begin{proof}
We mainly follow the argument of \cite[Lemma 4.7]{BBCK09}.
By Proposition \ref{p:domain} (ii),
\begin{equation}\label{e:deriv-lower}
\begin{split}
&H_{\varepsilon}'(t)\\
&=-{\cal E}^{\delta,(r),B_1}\left(u_r(t,\cdot), \frac{\Phi(\cdot)}{u_r^{\varepsilon}(t,\cdot)}\right)\\
&=-\iint_{B_1\times B_1}
(u_r^{\varepsilon}(t,y)-u_r^{\varepsilon}(t,x))
\frac{u_r^{\varepsilon}(t,x)\Phi(y)-u_r^{\varepsilon}(t,y)\Phi(x)}{u_r^{\varepsilon}(t,x)u_r^{\varepsilon}(t,y)}
r^{d+2} J^{(\delta)}(rx,ry)\,{\rm d}x\,{\rm d}y\\
&\quad -2\int_{B_1}\Phi(x)\left(r^{d+2}\int_{B_1^c}J^{(\delta)}(rx,ry)\,{\rm d}y\right)
\frac{u_r(t,x)}{u_r^{\varepsilon}(t,x)}\,{\rm d}x.
\end{split}
\end{equation}
Let $a=u_r^{\varepsilon}(t,y)/u_r^{\varepsilon}(t,x)$ and $b=\Phi(y)/\Phi(x)$.
Since $s+1/s-2\geq (\log s)^2$ for any $s>0$,
we have
\begin{align*}
&(u_r^{\varepsilon}(t,y)-u_r^{\varepsilon}(t,x))
\frac{u_r^{\varepsilon}(t,x)\Phi(y)-u_r^{\varepsilon}(t,y)\Phi(x)}{u_r^{\varepsilon}(t,x)u_r^{\varepsilon}(t,y)}\\
&=\Phi(x)\left(1-a+b-\frac{b}{a}\right)\\
&=\Phi(x)\left[(1-\sqrt{b})^2-\sqrt{b}\left(\frac{a}{\sqrt{b}}+\frac{\sqrt{b}}{a}-2\right)\right]\\
&\leq \Phi(x)\left[(1-\sqrt{b})^2-\sqrt{b}\left(\log\frac{a}{\sqrt{b}}\right)^2\right]\\
&=\left(\sqrt{\Phi(x)}-\sqrt{\Phi(y)}\right)^2-\sqrt{\Phi(x)\Phi(y)}
\left[\log\left(\frac{u_r^{\varepsilon}(t,y)}{\sqrt{\Phi(y)}}\right)
-\log\left(\frac{u_r^{\varepsilon}(t,x)}{\sqrt{\Phi(x)}}\right)\right]^2.
\end{align*}
Using this inequality with $0\leq u_r(t,x)/u_r^{\varepsilon}(t,x)\leq 1$, we obtain by \eqref{e:deriv-lower},
\begin{equation*}
\begin{split}
H_{\varepsilon}'(t)
&\geq -\iint_{B_1\times B_1}\left(\sqrt{\Phi(x)}-\sqrt{\Phi(y)}\right)^2r^{d+2}J^{(\delta)}(rx,ry)\,{\rm d}x\,{\rm d}y\\
&\quad+\iint_{B_1\times B_1}
\sqrt{\Phi(x)\Phi(y)}
\left[\log\left(\frac{u_r^{\varepsilon}(t,y)}{\sqrt{\Phi(y)}}\right)
-\log\left(\frac{u_r^{\varepsilon}(t,x)}{\sqrt{\Phi(x)}}\right)\right]^2
r^{d+2}J^{(\delta)}(rx,ry)\,{\rm d}x\,{\rm d}y\\
&\quad-2\int_{B_1}\Phi(x)\left(r^{d+2}\int_{B_1^c}J^{(\delta)}(rx,ry)\,{\rm d}y\right)\,{\rm d}x\\
&=-\iint_{{\mathbb R}^d\times {\mathbb R}^d}\left(\sqrt{\Phi(x)}-\sqrt{\Phi(y)}\right)^2r^{d+2}J^{(\delta)}(rx,ry)\,{\rm d}x\,{\rm d}y\\
&\quad+\iint_{B_1\times B_1}
\sqrt{\Phi(x)\Phi(y)}
\left[\log\left(\frac{u_r^{\varepsilon}(t,y)}{\sqrt{\Phi(y)}}\right)-\log\left(\frac{u_r^{\varepsilon}(t,x)}{\sqrt{\Phi(x)}}\right)\right]^2
r^{d+2}J^{(\delta)}(rx,ry)\,{\rm d}x\,{\rm d}y\\
&=:-({\rm I})+({\rm II}).
\end{split}
\end{equation*}

To give a lower bound of the last expression above,
we first show that
there exists a constant $C_1>0$,
which is independent of $\delta\in (0,1)$ and $\varepsilon\in (0,1)$,
such that
\begin{equation}\label{eq-1}
({\rm I})
\leq C_1\left(\int_{{\mathbb R}^d}|\nabla\sqrt{\Phi(x)}|^2\,{\rm d}x+\int_{B_1}\Phi(x)\,{\rm d}x\right).
\end{equation}
To do so, we write
\begin{align*}
({\rm I})
&=\iint_{\{0<|x-y|<1/r\}}\left(\sqrt{\Phi(x)}-\sqrt{\Phi(y)}\right)^2r^{d+2}J^{(\delta)}(rx,ry)\,{\rm d}x\,{\rm d}y\\
&\quad+\iint_{\{1/r\leq |x-y|<1\}}\left(\sqrt{\Phi(x)}-\sqrt{\Phi(y)}\right)^2r^{d+2}J^{(\delta)}(rx,ry)\,{\rm d}x\,{\rm d}y\\
&\quad +\iint_{\{|x-y|\geq 1\}}\left(\sqrt{\Phi(x)}-\sqrt{\Phi(y)}\right)^2r^{d+2}J^{(\delta)}(rx,ry)\,{\rm d}x\,{\rm d}y\\
&=:({\rm I})_1+({\rm I})_2+({\rm I})_3.
\end{align*}
By Assumption \ref{assum:jump} (ii) and \cite[(3.9)]{CKK08},
there exists a positive constant $c_1$,
which are independent of $\delta\in (0,1)$ and $r\geq 1$, such that
\begin{equation*}
\begin{split}
({\rm I})_1
&\leq \kappa_1 r^{d+2}\iint_{\{0<|x-y|<1/r\}}
\frac{\left(\sqrt{\Phi(x)}-\sqrt{\Phi(y)}\right)^2}{|rx-ry|^{d+\alpha_2}}\,{\rm d}x\,{\rm d}y
\leq c_1\int_{{\mathbb R}^d}|\nabla \sqrt{\Phi}(x)|^2\,{\rm d}x.
\end{split}
\end{equation*}
Since $6/(2-\alpha_2)>1$,
the function $\sqrt{\Phi(x)}=\sqrt{C_\Phi}(1-|x|^2)^{\frac{6}{2-\alpha_2}}{\bf1}_{B_1}(x)$ is Lipschitz continuous;
that is, there exists a positive constant $c_{\Phi}$ such that
$$
|\sqrt{\Phi(x)}-\sqrt{\Phi(y)}|\leq c_{\Phi}|x-y|
\quad \text{for any $x,y\in {\mathbb R}^d$}.
$$
We note that for any $\delta\in (0,1)$,  $J^{(\delta)}(rx,ry)=J(rx,ry)$ for $x,y\in \R^d$ and $r>1$ with $|rx-ry|\geq 1$.
Therefore, there exist positive constants  $c_{2i}$ ($i=1,2,3$),
which are independent of $r\geq 1$ and $\delta\in (0,1)$, such that
\begin{equation*}
\begin{split}
({\rm I})_2
&\leq c_{21} r^{d+2}\iint_{\{1/r\leq |x-y|<1\}}
{\left(\sqrt{\Phi(x)}-\sqrt{\Phi(y)}\right)^2}J(rx,ry)\,{\rm d}x\,{\rm d}y \\
&\leq c_{22}   r^{d+2}\int_{B_2} \left(\int_{\{1/r\leq |x-y|<1\}}{|x-y|^2}J(rx,ry)\,{\rm d}y\right)\,{\rm d}x\\
&\le \frac{c_{22}}{r^d} \int_{B_{2r}} \left(\int_{\{|x-y|\ge 1\}}{|x-y|^2}J(x,y)\,{\rm d}y\right)\,{\rm d}x\\
&\leq {c_{23}}=c_{23}\int_{B_1} \Phi(x)\,{\rm d}x,
\end{split}
\end{equation*}
where we used Assumption {\bf(B)} in the last inequality.
We also have
\begin{equation*}
\begin{split}
({\rm I})_3
&\leq  c_{31} r^{d+2}\int_{B_1}\left(\int_{\{|x-y|\geq 1\}}
 J(rx,ry)\,{\rm d}y \right)\,\d x\\
&=\frac{c_{31}r^2}{r^{d}} \int_{B_{r}}  \left(\int_{\{|x-y|\geq r\}}J(x,y)\,\d y\right)\,{\rm d}x\\
&\le c_{32}= c_{32} \int_{B_1}\Phi(x)\,{\rm d}x
\end{split}
\end{equation*}
for some positive constants $c_{3i}$ $(i=1,2)$,
which are independent of $r\geq 1$ and $\delta\in (0,1)$.
We thus arrive at (\ref{eq-1}).

We next show that there exist positive constants $c$ and $c'$,
which are independent of $\varepsilon\in(0,1)$, $\delta\in (0,1)$,
$x_1\in B_1\setminus{\cal N}_{\delta}$,
$t>0$ and $r\geq 1$,
such that
\begin{equation}\label{eq-2}
({\rm II})\geq -c+c'\int_{B_1}(\log u_r^{\varepsilon}(t,x)-H_{\varepsilon}(t))^2\Phi(x)\,{\rm d}x.
\end{equation}
To do so, we first prove that
\begin{equation}\label{e:l^2}
\int_{B_1}\left[\log\left(\frac{u_r^{\varepsilon}(t,x)}{\sqrt{\Phi(x)}}\right)\right]^2\,\d x<\infty.
\end{equation}
Since \eqref{e:upper1} implies that
$$u_r(t,x)=q_r^{\delta,B_1}(t,x,x_1)=r^dq^{\delta,B_r}(r^2t,rx,rx_1)
\leq c''r^d[(r^2t)^{-d/2}\vee (r^2t)^{-d/\alpha_1}],$$
we have
$$\varepsilon\leq u_r^{\varepsilon}(t,x)=u_r(t,x)+\varepsilon\leq c''r^d[(r^2t)^{-d/2}\vee (r^2t)^{-d/\alpha_1}]+\varepsilon$$
so that
$$0\leq (\log u_r^{\varepsilon}(t,x))^2
\leq[|\log \varepsilon|\vee |\log (c'' r^d((r^2t)^{-d/2}\vee (r^2t)^{-d/\alpha_1})+\varepsilon)|]^2.$$
Hence
$$\int_{B_1}(\log u_r^{\varepsilon}(t,x))^2\,{\rm d}x<\infty.$$
Noting that
$$\left[\log\left(\frac{u_r^{\varepsilon}(t,x)}{\sqrt{\Phi(x)}}\right)\right]^2
=\left(\log u_r^{\varepsilon}(t,x)-\log\sqrt{\Phi(x)}\right)^2
\leq 2(\log u_r^{\varepsilon}(t,x))^2+2(\log \sqrt{\Phi(x)})^2$$
and
$$\int_{B_1}(\log \sqrt{\Phi(x)})^2\,{\rm d}x<\infty,$$
we get \eqref{e:l^2}.

We next give a lower bound of $({\rm II})$. By \eqref{A:jump-kernel} and \eqref{eq-j-delta},
we have for all $r\ge 1$ and $x,y\in \R^d$,
$$
r^{d+2}J^{(\delta)}(rx,ry)
\geq r^{d+2}\frac{\kappa_1}{|rx-ry|^{d+\alpha_1}}{\bf 1}_{\{|x-y|<1/r\}}
=r^{2-\alpha_1}\frac{\kappa_1}{|x-y|^{d+\alpha_1}}{\bf 1}_{\{|x-y|<1/r\}}.
$$
Then by \eqref{e:l^2} and the weighted Poincar\'e inequality (\cite[Corollary 6]{DK13},
see also the argument in \cite[Theorem 4.1]{CKK11} and \cite[Proposition 3.2]{CKK08}),
we obtain
\begin{equation}\label{eq-2-1}
\begin{split}
({\rm II})
&\geq
r^{2-\alpha_1}\iint_{B_1\times B_1}
\sqrt{\Phi(x)\Phi(y)}
\left(\log\left(\frac{u_r^{\varepsilon}(t,y)}{\sqrt{\Phi(y)}}\right)-\log\left(\frac{u_r^{\varepsilon}(t,x)}{\sqrt{\Phi(x)}}\right)\right)^2\\
&\qquad\qquad \qquad\qquad \times
\frac{\kappa_1}{|x-y|^{d+\alpha_1}}{\bf 1}_{\{|x-y|<1/r\}}\,{\rm d}x\,{\rm d}y\\
&\geq c_4\int_{B_1}\left[\log\left(\frac{u_r^{\varepsilon}(t,x)}{\sqrt{\Phi(x)}}\right)
-\left(\int_{B_1}\log\left(\frac{u_r^{\varepsilon}(t,y)}{\sqrt{\Phi(y)}}\right)\Phi(y)\,{\rm d}y\right)\right]^2\Phi(x)\,{\rm d}x\\
&=c_{4}\int_{B_1}\left[\log\left(\frac{u_r^{\varepsilon}(t,x)}{\sqrt{\Phi(x)}}\right)
-\left(H_{\varepsilon}(t)-\frac{1}{2}\int_{B_1}\Phi(y)\log \Phi(y)\,\d y\right)\right]^2\Phi(x)\,{\rm d}x
\end{split}
\end{equation}
for some positive constant $c_{4}=c_{4}(\kappa_1,d,\alpha_1, \Phi)$,
which is independent of $\delta\in (0,1)$,
$x_1\in B_1\setminus{\cal N}_{\delta}$,
$t>0$, $r\geq 1$ and $\varepsilon\in (0,1)$.
Moreover, since
\begin{equation*}
\begin{split}
(\log u_r^{\varepsilon}(t,x)-H_{\varepsilon}(t))^2
&\leq 2\left[\log\left(\frac{u_r^{\varepsilon}(t,x)}{\sqrt{\Phi(x)}}\right)
-\left(H_{\varepsilon}(t)-\frac{1}{2}\int_{B_1}\Phi(y)\log\Phi(y)\,\d y\right)\right]^2\\
&\quad
+2\left(\frac{1}{2}\log\Phi(x)-\frac{1}{2}\int_{B_1}\Phi(y)\log\Phi(y)\,{\rm d}y\right)^2,
\end{split}
\end{equation*}
the last expression in \eqref{eq-2-1} is greater than
$$
\frac{c_{4}}{2}\int_{B_1}\left(\log u_r^{\varepsilon}(t,x)-H_{\varepsilon}(t)\right)^2\Phi(x)\,{\rm d}x-c_{5}
$$
for
$$c_{5}=\frac{c_{4}}{4}
\int_{B_1} \left(\log\Phi(x)-\int_{B_1}\Phi(y)\log\Phi(y)\,{\rm d}y\right)^2
\Phi(x)\,{\rm d}x,$$
whence \eqref{eq-2} follows.

Combining \eqref{eq-1} with \eqref{eq-2}, we have \eqref{eq-g'}.
The proof is complete. \end{proof}

\begin{lem}\label{lem-dist}
Under Assumption {\bf (B)},
there exist constants $t_0\in(0,1)$ small enough and $c_*=c_*(t_0)\geq 1$ such that
the following assertions hold.
\begin{enumerate}
\item For all $\delta\in (0,1)$, $r\geq c_*$, $t\in[t_0/8,2t_0]$ and $x\in \R^d\backslash \N_\delta$,
$$\Pp^{x}\left(|Y_t^{\delta,(r)}-Y_0^{\delta,(r)}|>\frac{1}{4}\right)\leq \frac{1}{12}.$$
\item For all $\delta\in (0,1)$, $r\geq c_*$, $t\in[t_0/8, t_0]$ and $x_1\in B_{1/2}\backslash \N_\delta$,
$$\int_{B(x_1,1/4)}u_r(t,x)\,\d x\geq \frac{3}{4}.$$
\end{enumerate}
\end{lem}

\begin{proof}
(i) \ By \eqref{eq-heat-scale} and the change of variables,
we have for all $t>0$ and $x\in \R^d\backslash \N_\delta$,
\begin{equation*}
\begin{split}
\Pp^{x}\left(|Y_t^{\delta,(r)}-Y_0^{\delta,(r)}|>\frac{1}{4}\right)
&=\int_{\{|y-x|\geq 1/4\}}q_r^{\delta}(t,x,y)\,{\rm d}y\\
&=r^d\int_{\{|y-x|\geq 1/4\}}q^{\delta}(r^2t,rx,ry)\,{\rm d}y
=\int_{\{|y-rx|\geq r/4\}}q^{\delta}(r^2t,rx,y)\,{\rm d}y\\
&=\int_{\{|y-rx|\geq r/4,
|y-rx|^2\ge r^2 t \}}
q^{\delta}(r^2t,rx,y)\,{\rm d}y\\
&\quad +\int_{\{|y-rx|\geq r/4,
r^2 t> |y-rx|^2\}}
q^{\delta}(r^2t,rx,y)\,{\rm d}y\\
&=:({\rm I})+({\rm II}).
\end{split}
\end{equation*}
Since the jumping kernel $J^{(\delta)}(x,y)$ fulfills Assumption {\bf(B)},
we see by Proposition \ref{thm-upper-bound0} that
there are constants $c_i \,(i=1,2)>0$ and $t_1>0$ (both are independent of $\delta\in(0,1)$)
such that for all $r^2t\ge t_1$ and $x\in \R^d\backslash \N_\delta$,
\begin{equation*}
\begin{split}
({\rm I})
&\leq \int_{\{|y-rx|\geq r/4, |y-rx|^2\ge r^2 t\}}
\frac{c_1 r^2 t}{|y-rx|^{d+2}}\,{\rm d}y\leq c_1 r^2 t \int_{\{|y-rx|\geq r/4\}}\frac{{\rm d}y}{|y-rx|^{d+2}}
=c_2t.
\end{split}
\end{equation*}
On the other hand, if $t\leq 1/16$, then  $r^2 t\leq r^2/16$,
and so
$({\rm II})=0.$
Therefore, if we choose $t_2>0$ small enough such that
\begin{equation*}\label{eq-t_0}
t_2\leq \frac{1}{32} \quad \text{and} \quad
c_2t_2\leq \frac{1}{24},
\end{equation*}
then for any $r\geq \sqrt{{8t_1}/{t_2}}$ and $t\in [t_2/8, 2t_2]$,
$$\Pp^{x}\left(|Y_t^{\delta,(r)}-Y_0^{\delta,(r)}|>\frac{1}{4}\right)\leq \frac{1}{12}.$$
The desired assertion follows by taking $t_0=t_2$ and $c_*=1\vee \sqrt{8t_1/t_2}$.

\noindent
(ii) \ For an open subset $D$ of ${\R^d}$,
let $\tau_D^{Y^{\delta,(r)}}$ be the exit time of $Y^{\delta,(r)}$ from $D$.
Since $q_r^{\delta,B_1}(t,x,x_1)=q_r^{\delta,B_1}(t,x_1,x)$,
\begin{equation}\label{e:lower-u}
\begin{split}
\int_{B(x_1,1/4)}u_r(t,x)\,\d x
&=\int_{B(x_1,1/4)}q_r^{\delta,B_1}(t,x,x_1)\,\d x\\
&=\int_{B(x_1,1/4)}q_r^{\delta,B_1}(t,x_1,x)\,\d x\\
&=\Pp^{x_1}(|Y^{\delta,(r), B_1}_t-x_1|<1/4)\\
&=\Pp^{x_1}\left(|Y^{\delta,(r)}_t-x_1|<1/4, t<\tau_{B_1}^{Y^{\delta,(r)}}\right).
\end{split}
\end{equation}
Noting that
\begin{equation*}
\begin{split}
1&=\Pp^{x_1}\left(|Y^{\delta,(r)}_t-x_1|<1/4, t<\tau_{B_1}^{Y^{\delta,(r)}}\right)
+\Pp^{x_1}\left(|Y^{\delta,(r)}_t-x_1|<1/4, \tau_{B_1}^{Y^{\delta,(r)}}\leq t\right)\\
&\quad +\Pp^{x_1}\left(|Y^{\delta,(r)}_t-x_1|\geq 1/4\right)\\
&\leq \Pp^{x_1}\left(|Y^{\delta,(r)}_t-x_1|<1/4, t<\tau_{B_1}^{Y^{\delta,(r)}}\right)
+\Pp^{x_1}\left(\tau_{B_1}^{Y^{\delta,(r)}}\leq t\right)
+\Pp^{x_1}\left(|Y^{\delta,(r)}_t-x_1|\geq 1/4\right),
\end{split}
\end{equation*}
we get by \eqref{e:lower-u},
\begin{equation}\label{e:lower-u2}
\int_{B(x_1,1/4)}u_r(t,x)\,\d x
\geq 1-\Pp^{x_1}\left(\tau_{B_1}^{Y^{\delta,(r)}}\leq t\right)
-\Pp^{x_1}\left(|Y^{\delta,(r)}_t-x_1|\geq 1/4\right).
\end{equation}

Let $X=(\{X_t\}_{t\geq 0}, \{{\Pp}^x\}_{x\in {\R^d}})$
be the strong Markov process on $\R^d$ and $\tau_D$ the exit time of $X$ from $D$.
Then by the same way as in \cite[(2.18)]{BGK09},
the strong Markov property implies that for any $x\in \R^d$, $t>0$ and $r>0$,
\begin{equation}\label{e:ff7}
\begin{split}
\Pp^x (\tau_{B(x,r)}\le t)\le & \Pp^x(\tau_{B(x,r)}\le t, |X_{2t}-x|\le r/2)+\Pp^x(|X_{2t}-x|\ge r/2)\\
\le&\Pp^x(\tau_{B(x,r)}\le t, |X_{2t}-X_{\tau_{B(x,r)}}|\ge r/2)+ \Pp^x(|X_{2t}-x|\ge r/2)\\
\le&\sup_{s\le t, |z-x|\ge r}\Pp^z(|X_{2t-s}-z|\ge r/2)+ \Pp^x(|X_{2t}-x|\ge r/2)\\
\le& 2 \sup_{s\in [t,2t], z\in \R^d}\Pp^z(|X_{s}-z|\ge r/2).
\end{split}
\end{equation}
Applying it to $\{Y^{\delta,(r)}_t\}_{t\geq 0}$, we see that for any $x_1\in B_{1/2}\setminus{\cal N}_{\delta}$,
$$\Pp^{x_1}(\tau_{B_1}^{Y^{\delta,(r)}}\leq t)
\leq \Pp^{x_1}(\tau_{B(x_1,1/2)}^{Y^{\delta,(r)}}\leq t)
\leq 2 \sup_{s\in [t,2t], z\in \R^d}\Pp^z(|Y^{\delta,(r)}_{s}-z|\ge 1/4).$$
Then by (i), we obtain for any $x_1\in B_{1/2}\setminus{\cal N}_{\delta}$ and $t\in [t_0/8,t_0]$,
\begin{equation*}
\begin{split}
&\Pp^{x_1}\left(\tau_{B_1}^{Y^{\delta,(r)}}\leq t\right)
+\Pp^{x_1}\left(|Y^{\delta,(r)}_t-x_1|\geq 1/4\right)\\
&\leq 2 \sup_{s\in [t,2t], z\in \R^d}\Pp^z(|Y^{\delta,(r)}_{s}-z|\ge 1/4)
+\Pp^{x_1}\left(|Y^{\delta,(r)}_t-x_1|\geq 1/4\right)\\
&\leq 2\cdot\frac{1}{12}+\frac{1}{12}=\frac{1}{4}.
\end{split}
\end{equation*}
Hence the proof is complete by \eqref{e:lower-u2}.
\end{proof}

Now, we are in position to give the proof of Proposition \ref{thm-lower-1}.

\begin{proof}[Proof of Proposition $\ref{thm-lower-1}$]
Let $t_0\in(0,1)$ and $c_*\geq 1$ be the same constants as in Lemma \ref{lem-dist}.
We first prove that there exists a positive constant $c=c(t_0)$ such that
for all $\delta\in(0,1)$,  $r\geq c_*$, $x_1\in B_{1/2}\backslash \N_\delta$
and $t_1\in [t_0/4, t_0]$,
$$\int_{B_1}\Phi(y)\log q_r^{\delta,B_1}(t_1,y,x_1)\,\d y\geq -c.$$

Our approach here is similar to that of \cite[Lemmas 3.3.1--3.3.3]{Da89} and \cite[Proof of Theorem 2.5]{Fo09}.
Fix $\varepsilon\in (0,1)$, $\delta\in (0,1)$, $x_1\in B_{1/2} \backslash \N_\delta$,  $r\geq c_*$ and $t\in [t_0/8, t_0]$.
Let  $K$ be a constant such that $|B(x_1,1/4))|e^{-K}=1/4$, and define
$$D_t^{\varepsilon}:=\left\{x\in B\left(x_1,{1}/{4}\right) \mid u_r^{\varepsilon}(t,x)\geq e^{-K}\right\}.$$
Then
$$
\int_{B(x_1,1/4)\setminus D_t^{\varepsilon}}u_r(t,x)\,{\rm d}x
\leq
\int_{B(x_1,1/4)\setminus D_t^{\varepsilon}}u_r^{\varepsilon}(t,x)\,{\rm d}x
\leq e^{-K}|B(x_1,1/4)|
=\frac{1}{4}.
$$
Since  $r\geq 1$ and $t\leq 1$ by assumption,
we get from \eqref{e:upper1} that
\begin{equation}\label{e:upper-scale}
\begin{split}
u_r(t,x)=r^dq^{\delta, B_r}(r^2t,rx,rx_1)&\leq r^dq^{\delta}(r^2t,rx,rx_1)\\
&\leq c_{1}r^d\left((r^2 t)^{-d/2}\vee (r^2 t)^{-d/\alpha_1}\right)\leq c_{1}t^{-d/\alpha_1},
\end{split}
\end{equation}
where $c_{1}$ is a positive constant
independently of $\delta\in (0,1)$, $r\geq 1$ and $x, x_1\in B_{1/2}\backslash \N_\delta$.
Then
$$
\int_{D_t^{\varepsilon}}u_r(t,x)\,{\rm d}x
\leq \frac{c_{1}}{t^{d/\alpha_1}}|D_t^{\varepsilon}|.
$$
Combining all the estimates above with  Lemma \ref{lem-dist} (ii), we have
$$
\frac{3}{4}
\leq \int_{B(x_1,1/4)}u_r(t,x)\,{\rm d}x
=\int_{D_t^{\varepsilon}}u_r(t,x)\,{\rm d}x+\int_{B(x_1,1/4)\setminus D_t^{\varepsilon}}u_r(t,x)\,{\rm d}x
\leq \frac{c_{1}}{t^{d/\alpha_1}}|D_t^{\varepsilon}|+\frac{1}{4};
$$
that is,
$$|D_t^{\varepsilon}|\geq \frac{t^{d/\alpha}}{2c_{1}}
\geq \frac{1}{2c_{1}}\left(\frac{t_0}{8}\right)^{d/\alpha}
\quad \text{for all $t\in [t_0/8,t_0]$}.$$
Furthermore, by following the argument in \cite[p.851--852]{CKK08} and using Lemma \ref{lem-g'},
there exists a positive constant $c_{2}=c_{2}(t_0)$,
which is independent of $\varepsilon\in (0,1)$, $\delta\in(0,1)$,  $r\geq c_*$ and $x_1\in B_{1/2}\backslash \N_\delta$,
such that for any $t_1\in [t_0/4, t_0]$,
\begin{equation}\label{e:lower-h}
H_{\varepsilon}(t_1)=\int_{B_1}\Phi(y)\log u_r^{\varepsilon}(t_1,y)\,\d y\geq -c_{2}.
\end{equation}

Note that if  $0<\varepsilon<1\wedge (2c_1/t_0^{d/\alpha_1})$,
then by \eqref{e:upper-scale},
$$\frac{\varepsilon t_1^{d/\alpha_1}}{2c_1}
\leq \frac{t_1^{d/\alpha_1}}{2c_1}u_r^{\varepsilon}(t_1,y)
= \frac{t_1^{d/\alpha_1}}{2c_1}(u_r(t_1,y)+\varepsilon)
\leq \frac{1}{2}+\frac{t_0^{d/\alpha_1}\varepsilon}{2c_1}\leq 1.$$
Therefore, by the monotone convergence theorem,
\begin{equation*}
\begin{split}
\int_{B_1}\Phi(y)\log \left(\frac{t_1^{d/\alpha_1}}{2c_1}u_r^{\varepsilon}(t_1,y)\right)\,\d y
\rightarrow \int_{B_1}\Phi(y)\log \left(\frac{t_1^{d/\alpha_1}}{2c_1}u_r(t_1,y)\right)\,\d y
\quad (\varepsilon\downarrow0).
\end{split}
\end{equation*}
Then by letting $\varepsilon\downarrow0$ in \eqref{e:lower-h}, we get
$$\int_{B_1}\Phi(y)\log q_r^{\delta,B_1}(t_1,y,x_1)\,\d y=\int_{B_1}\Phi(y)\log u_r(t_1,y)\,\d y\geq -c_{2},$$
which is the desired inequality.

We next discuss the lower bound of $q^{\delta}(t,x,y)$.
By Jensen's inequality,
there exists a positive constant
$c_{3}=c_{3}(t_0,\Phi)$
such that for all $\delta\in (0,1)$, $r\geq c_*$, $t_1\in [t_0/4, t_0]$ and $x_0,x_1\in B_{1/2}\backslash \N_\delta$,
\begin{equation*}
\begin{split}
\log q_r^{\delta,B_1}(2t_1,x_0,x_1)
&=\log\left(\int_{B_1}q_r^{\delta,B_1}(t_1,x_0,y)q_r^{\delta, B_1}(t_1,y,x_1)\,{\rm d}y\right)\\
&\geq
\log\left(\int_{B_1}q_r^{\delta,B_1}(t_1,x_0,y)q_r^{\delta, B_1}(t_1,y,x_1)\Phi(y)\,{\rm d}y\right)
-\log \|\Phi\|_{\infty}\\
&\geq  \int_{B_1}\log \left(q_r^{\delta,B_1}(t_1,x_0,y)q_r^{\delta, B_1}(t_1,y,x_1)\right)\Phi(y)\,{\rm d}y
-\log \|\Phi\|_{\infty}\\
&=\int_{B_1} \Phi(y)\log q_r^{\delta,B_1}(t_1,x_0,y)\,{\rm d}y
+\int_{B_1}\Phi(y) q_r^{\delta, B_1}(t_1,y,x_1)\,{\rm d}y\\
& \quad -\log \|\Phi\|_{\infty}\\
&\geq -c_{3};
\end{split}
\end{equation*}
that is,
\begin{equation}\label{e:heat-lower}
q_r^{\delta,B_1}(t,x_0,x_1)\geq e^{-c_{3}} \quad
\text{for all $t\in [t_0/2,2t_0]$}.
\end{equation}

As we see from the proof of Lemma \ref{lem-dist},
the positive constant $t_0$ can be arbitrary small.
In what follows, without loss of generality we may and can assume that $0<t_0<1/4$.
Then for any $t\in [1/2,2]$, there exists a positive integer $k_t\geq 1$ such that
$t-k_t t_0/2\in [t_0/2,2t_0]$. In fact,
\begin{equation}\label{e:k_t}
0<\frac{1}{t_0}-4\leq \frac{t-2t_0}{t_0/2}\leq k_t\leq\frac{t-t_0/2}{t_0/2}\leq \frac{4}{t_0}-1
\end{equation}
and
$$\frac{t-t_0/2}{t_0/2}-\frac{t-2t_0}{t_0/2}=3.$$
By the semigroup property and \eqref{e:heat-lower},
we have for any $t\in [1/2,2]$ and $x_0,x_1\in B_{1/2}\backslash \N_\delta$,
\begin{equation*}
\begin{split}
r^d q^{\delta,B_r}(r^2t,rx_0,rx_1)
&=q_r^{\delta,B_1}(t,x_0,x_1)\\
&=\int_{B_1}q_r^{\delta,B_1}(t-t_0/2,x_0,z_1)q_r^{\delta,B_1}(t_0/2,z_1,x_1)\,\d z_1\\
&\geq \int_{B_{1/2}}q_r^{\delta,B_1}(t-t_0/2,x_0,z_1)q_r^{\delta,B_1}(t_0/2,z_1,x_1)\,\d z_1\\
&\geq e^{-c_3} \int_{B_{1/2}}q_r^{\delta,B_1}(t-t_0/2,x_0,z_1)\, \d z_1.
\end{split}
\end{equation*}
By the same way, the last term above is equal to
\begin{equation*}
\begin{split}
&e^{-c_3} \int_{B_{1/2}}
\left(\int_{B_1}q_r^{\delta,B_1}(t-2\cdot t_0/2,x_0,z_2)q_r^{\delta,B_1}(t_0/2,z_2,z_1)\,\d z_2\right)\d z_1\\
&\geq e^{-2c_3} \int_{B_{1/2}}\left(\int_{B_{1/2}}q_r^{\delta,B_1}(t-2\cdot t_0/2,x_0,z_2)\,\d z_2\right)\d z_1.
\end{split}
\end{equation*}
By repeating this procedure and using \eqref{e:k_t},
there exists a positive constant $c_{4}=c_{4}(t_0,\Phi)$ such that
for all $\delta\in(0,1)$, $r\geq c_*$, $t\in [1/2,2]$  and $x_0,x_1\in B_{1/2}\backslash \N_\delta$,
\begin{equation}\label{e:iteration}
\begin{split}
r^d q^{\delta,B_r}(r^2t,rx_0,rx_1)
&\geq e^{-k_t c_3} \int_{B_{1/2}}\cdots \int_{B_{1/2}}q_r^{\delta,B_1}(t-k_t t_0/2,x_0,z_{k_t})\,\d z_{k_t}\cdots\,d z_1\\
&\geq e^{-(k_t+1) c_3}|B_{1/2}|^{k_t}
\geq c_4,
\end{split}
\end{equation} where $c_4$ is independent of $t$.

By taking $t=1$ in \eqref{e:iteration},
we find that for all $\delta\in(0,1)$, $r\geq c_*$ and $x_0,x_1\in B_{1/2}\backslash \N_\delta$,
$$q^{\delta,B_r}(r^2,rx_0,rx_1)\geq \frac{c_{4}}{r^d}.$$
Letting $r=\sqrt{t}$ in the estimate above,
we have for any  for $t\geq c_*^2$ and $x_0,x_1\in B_{1/2} \backslash \N_\delta$,
$$q^{\delta,B_{\sqrt{t}}}(t,\sqrt{t}x_0,\sqrt{t}x_1)\geq \frac{c_{4}}{t^{d/2}};$$
that is,
$$q^{\delta,B_{\sqrt{t}}}(t, x_0, x_1)\geq \frac{c_{4}}{t^{d/2}},\quad x_0,x_1\in B_{\sqrt{t}/2}\backslash \N_\delta.$$
By the space-uniformity of ${\mathbb R}^d$,
we can replace the center of any ball by $z_0\in{\mathbb R}^d$ in the argument above.
Hence for any $t\geq c_*^2$, $z_0\in \R^d$ and $x,y\in B(z_0,\sqrt{t}/2)\backslash \N_\delta$,
$$q^{\delta}(t,x,y)\geq q^{\delta,B(z_0,\sqrt{t})}(t, x, y)\geq \frac{c_{4}}{t^{d/2}}.$$
Note that for any $x,y\in \R^d$ with $|x-y|^2\leq t$,
there exists a point $z_0\in {\mathbb R}^d$ such that
$x,y\in B(z_0,\sqrt{t}/2)$.
Therefore, our assertion is valid  for $t\geq c_*^2$.
\end{proof}

At the end of this section, we present two-sided heat kernel estimates for jump processes,
upper bounds of which have been established in Corollary \ref{c:two}.
\begin{cor}\label{c:two1}
Assume that there is a constant $\varepsilon>0$ such that for all $x,y\in \R^d$ with $|x-y|\ge 1$,
$$J(x,y)\asymp \frac{1}{|x-y|^{d+2+\varepsilon}}.$$
Then, there exist positive constants $t_0\ge 1$, $\theta_0>0$ and $c_0$ such that for all $t\ge t_0$,
$$p(t,x,y)\asymp
\begin{cases}
\displaystyle \frac{1}{t^{d/2}},& t\geq |x-y|^2,\\
\displaystyle \frac{1}{t^{d/2}}\exp\left(-\frac{c_0|x-y|^2}{t}\right),
& \displaystyle  \frac{\theta_0|x-y|^2}{\log(1+|x-y|)}\leq t\leq |x-y|^2,\\
\displaystyle \frac{1}{|x-y|^{d+2+\varepsilon}}, & \displaystyle t\leq \frac{\theta_0|x-y|^2}{\log(1+|x-y|)}.
\end{cases}$$
Here we note that the constants $c_0$ and $\theta_0$ in the formula above
should be different for upper and lower bounds. \end{cor}
\begin{proof}
The upper bound estimates have been proved in Corollary \ref{c:two},
so we need verify lower bounds. According to Theorem \ref{T:lower}, we have got the first two cases, i.e.\
$t\geq |x-y|^2$ and $\frac{\theta_0|x-y|^2}{\log(1+|x-y|)}\leq t\leq |x-y|^2$.
Then, the proof is complete, if we prove that there exist
 constants $t_0\ge 1$ and $c_1, c_2>0$ such that for all $t_0\le t\le c_1|x-y|^2$,
\begin{equation}\label{e:ff5}p(t,x,y)\ge \frac{c_2}{|x-y|^{d+2+\varepsilon}}.\end{equation}

(1) First, we claim that there are positive constants $c_0$ and $t_0$ such that
for all $t\ge t_0$ and $x\in \R^d \setminus\N$,
\begin{equation}\label{e:ff6}
\Pp^x(\tau_{B(x,c_0\sqrt{t})}\le t)\le 1/2.
\end{equation}
Indeed, we recall \eqref{e:ff7}: for any $x\in \R^d\setminus \N$ and $t,r>0$,
\begin{equation}\label{e:ff7-1}
\Pp^x (\tau_{B(x,r)}\le t)
\le 2 \sup_{s\le t, z\in \R^d}\Pp^z(|X_{2t-s}-z|\ge r/2). \end{equation}
Now, according to upper bound estimates for $p(t,x,y)$ in Corollary \ref{c:two},
there is a constant $t_0>0$ such that for all $t\ge t_0$, $r^2\ge t$ and $x\in \R^d\setminus \N$,
\begin{align*}\Pp^x(|X_t-x|\ge r)&\le c_1\left(\int_{\{|y-x|\ge r\}} t^{-d/2} \exp\left(-c_2|x-y|^2/t\right)\,\d y+\int_{\{|y-x|\ge r\}}
 \frac{t}{|x-y|^{d+2+\varepsilon}}\,\d y\right)\\
 &\le c_3\left(\int_{r^2/t}^\infty e^{-c_2s}s^{d/2-1}\,\d s+\int_r^\infty \frac{t}{s^{3+\varepsilon}}\,\d s\right)\\
 &\le c_4\left(e^{-c_5 r^2/t}+  \frac{t}{r^{2+\varepsilon}}\right). \end{align*} In particular, taking $r\ge c_6t^{1/2}$ for some $c_6$ large enough, we find that
 $$\Pp^x(|X_t-x|\ge r)\le 1/4.$$ This along with \eqref{e:ff7-1} yields \eqref{e:ff6}.

 (2) Next, we will use the approach of \cite[Section 4.4]{CZ16}. Fix $t\ge t_0$ and $x,y\in \R^d\setminus \N$ with $|x-y|\ge 4c_0 t^{1/2},$
 where $c_0$ is the constant in \eqref{e:ff6}. It follows from the
Chapman-Kolmogorov equation and Theorem \ref{T:lower} that
\begin{align*}
p(2t,x,y)=&\int_{\R^d} p(t,x,z)p(t,z,y)\,\d z\\
\ge &\left(\inf_{|z-y|\le 2c_0t^{1/2}} p(t,z,y)\right) \int_{\{|y-z|\le 2c_0t^{1/2}\}} p(t,x,z)\,\d z\\
\ge &c_1 t^{-d/2}\Pp^x(X_t\in B(y,2c_0t^{1/2})).
\end{align*}
For any $x\in \R^d$ and $r>0$, define
$$\sigma_{B(x,r)}=\inf\{t>0:X_t\in B(x,r)\}.$$
By the strong Markov property,
\begin{align*}&\Pp^x(X_t\in B(y,2c_0t^{1/2}))\\
&\ge \Pp^x\left(\sigma_{B(y,c_0t^{1/2})}\le t/2; \sup_{s\in\big[\sigma_{B(y,c_0t^{1/2})},t\big]}
|X_s-X_{\sigma_{B(y,c_0t^{1/2})}}|\le c_0t^{1/2}\right)\\
&\ge \Pp^x\left(\sigma_{B(y,c_0t^{1/2})}\le t/2\right)\inf_{z\in B(y,c_0t^{1/2})}\Pp^z(\tau_{B(z,c_0t^{1/2})}>t)\\
&\ge \frac{1}{2}\Pp^x\left(\sigma_{B(y,c_0t^{1/2})}\le t/2\right),
\end{align*} where we used \eqref{e:ff6} in the last inequality.
Furthermore, by the L\'evy system formula (see \cite[p.151]{BGK09} and \cite[Appendix A]{CK08}) and the fact that $|x-y|\ge 4c_0 t^{1/2},$
\begin{align*}\Pp^x\left(\sigma_{B(y,c_0t^{1/2})}\le t/2\right)
&\ge \Pp^x(X_{(t/2)\wedge \tau_{B(x,c_0t^{1/2})}}\in B(y, c_0t^{1/2}))\\
&\ge c_2\Ee^x\left(\int_0^{(t/2)\wedge \tau_{B(x,c_0t^{1/2})}} \int_{B(y, c_0t^{1/2})} \frac{\d z}{|X_s-z|^{d+2+\varepsilon}}\,\d s\right)\\
&\ge c_3 t^{d/2+1}\Pp^x(\tau_{B(x,c_0t^{1/2})}\ge t/2) \frac{1}{|x-y|^{d+2+\varepsilon}}\\
&\ge c_4t^{d/2+1}\frac{1}{|x-y|^{d+2+\varepsilon}},\end{align*}
where in the third inequality we used the facts that $|x-y|\ge 4c_0 t^{1/2}$, and
for all $s\in (0, (t/2)\wedge \tau_{B(x,c_0t^{1/2})})$ and $z\in B(y,  c_0t^{1/2})$,
$$|X_s-z|\le |X_s-x|+|x-y|+|y-z|\le 2 c_0t^{1/2}+|x-y|\le 2 |x-y|;$$ and the last inequality follows from \eqref{e:ff6}.
Combining all the inequalities above, we find that
$t\ge t_0$ and $x,y\in \R^d\setminus \N$ with $|x-y|\ge 4c_0 t^{1/2},$
$$p(2t,x,y)\ge \frac{c_4 t}{|x-y|^{d+2+\varepsilon}},$$ which proves \eqref{e:ff5}. \end{proof}

\section{Proof of Theorem \ref{main}}

\begin{proof}[Proof of Theorem $\ref{main}$]
Throughout this proof, we set $\psi(r)=\sqrt{r\log\log r}$.
Recall that $\tau_{B(x,r)}=\inf\{t>0:X_t\notin B(x,r)\}$
for any $x\in \R^d$ and $r>0$.

(1) In this case, $\phi(s)=\log^{1+\varepsilon}(e+s)$ and so
$c^{-1}\log^{\varepsilon}(e+s)\leq \Phi(s)\leq c\log^{\varepsilon}(e+s)$ for some constant
$c\ge1$.
We follow the proof of
\cite[Theorem 3.1(1)]{SW17} first.
Setting $t_k=2^k$, we have for all $k\ge 2$ and $x\in \R^d\setminus \N$,
\begin{equation}\begin{split}\label{e:ff1}\Pp^x(|X_s-x|&\ge C_0\psi(s)\hbox{ for some }s\in [t_{k-1},t_k])\\
&\le \Pp^x(\sup_{s\in [t_{k-1},t_k]}|X_s-x|\ge C_0\psi(t_{k-1}))\le \Pp^x(\tau_{B(x,C_0\psi(t_{k-1}))}\le t_k)\\
&\le 2 \sup_{s\le t_k, z\in \R^d}\Pp^z(|X_{t_{k+1}-s}-z|\ge C_0\psi(t_{k-1})/2),
\end{split}\end{equation}
where in the last inequality we used \eqref{e:ff7}.

For any $\kappa\ge 1$, let $\theta_0$ be the constant in Theorem \ref{thm-upper-bound}.
We choose $\theta_0^*>C$ large enough such that, if $r\ge \theta_0^*\psi(t)$, then
$t\le \frac{\theta_0 r^2}{\log \Phi(r)};$ if $r\le\theta_0^*\psi(t)$, then $ t\ge \frac{\theta'_0 r^2}{\log \Phi(r)}$ for some constant $\theta_0'\in (0,1)$.
Below, we fix this $\kappa$ and $\theta_0^*$, and let $\delta>0$ first. For any $x\in \R^d \setminus \N$ and $t, C>0$ large enough, according to Theorem \ref{thm-upper-bound}, Remark \ref{rem-upper-bound}(ii) and
Proposition \ref{thm-upper-bound3} (with $\delta=1/2$),
\begin{align*} &\Pp^x(|X_t-x|\ge C \psi(t))\\
&=\int_{\{|y-x|\ge C\psi(t)\}} p(t,x,y)\,\d y\\
&\le \frac{c_1}{t^{d/2}}\int_{\{C\psi(t)\le |y-x|\le  \theta^*_0\psi(t)\}} \exp\left(-\frac{c_2|x-y|^2}{t}\right) \,\d y\\
&\quad+c_3\int_{\{\theta^*_0\psi(t)\le |y-x|\le c_4\sqrt{t\log^{1+\delta}t}\}} \left( t^{-d/2} \frac{1}{\log^{\kappa \varepsilon/8} |x-y|}+
\frac{t}{|x-y|^{d+2}\log^{1+\varepsilon}|x-y|}\right)\,\d y\\
&\quad+c_5 \int_{\{|y-x|\ge c_4\sqrt{t\log^{1+\delta}t}\}} \frac{ t}{|x-y|^{d+2}\log^{(d+2)/4}\log\log (1+|x-y|)} \,\d y\\
&=:I_1+I_2+I_3,
\end{align*} where the constants $c_i (i=1,\cdots, 5)$ may depend on $\kappa$ and $\delta$.
First, it holds that
\begin{align*}I_2\le& c_{21}\left[(t\log^{1+\delta}t)^{d/2} \left (t^{-d/2}\log^{-\kappa \varepsilon/8} t  \right)+
\int_{\theta^*_0\psi(t)}^\infty \frac{t}{r^3\log^{1+\varepsilon}r}\,\d r \right]\\
\le & c_{22}\left[\log^{-((\kappa \varepsilon/8)-((1+\delta)d/2))} t +
\frac{1}{\log^{1+\varepsilon} t} \right]. \end{align*} Taking $\kappa\ge1$ large enough such that
$ \kappa \varepsilon/8\ge (1+\delta)d/2+1+\varepsilon,$ we find that
$$I_2\le  \frac{c_{23}}{\log^{1+\varepsilon} t}.$$ Second, we fix $\kappa$ as above. We find that
\begin{align*}I_1\le&\frac{c_{11}}{t^{d/2}}\int_{\{|y-x|\ge C\psi(t)\}} \exp\left(-\frac{c_2|x-y|^2}{t}\right) \,\d y\\
\le& c_{12}\int_{C^2\log\log t}^\infty\exp(-c_2s) s^{d/2-1}\,\d s\le c_{13}(\log t)^{-C^2 c_2/2}, \end{align*} where $c_2$ depends on $\kappa$ above. Choosing $C>1$ large enough such that
$C^2 c_2/2\ge 1+\varepsilon$, we get that $$I_1\le  \frac{c_{14}}{\log^{1+\varepsilon} t}.$$
Third, it is easy to see that
$$I_3\le  \frac{c_{31}}{\log^{1+\delta} t}.$$ In particular, letting $\delta=\varepsilon$,
$$I_3\le  \frac{c_{32}}{\log^{1+\varepsilon} t}.$$
By all the estimates above, we obtain that there is a constant $C_1>0$ such that
for any $x\in \R^d \setminus \N$ and $t, C>0$ large enough,
\begin{equation}\label{e:ff2} \Pp^x(|X_t-x|\ge C \psi(t))\le \frac{C_1}{\log^{1+\varepsilon} t}.\end{equation}

According to \eqref{e:ff1} and \eqref{e:ff2}, we know that there is a constant $C_2>0$ such that for all $k\ge 2$, $C_0>0$ large enough and $x\in \R^d\setminus \N$,
$$\Pp^x(|X_s-x|\ge C_0\psi(s)\hbox{ for some }s\in [t_{k-1},t_k])\le \frac{C_2}{k^{1+\varepsilon}}.$$ This together with the Borel-Cantelli lemma proves the first desired assertion.

(2) For any $c>0$ and $k\ge 1$, set $t_k=2^k$ and
$$B_k=\{|X_{t_{k+1}}-X_{t_k}|\ge c\psi(t_{k-1})\}.$$
Denote by $({F}_t)_{t\ge0}$ the natural filtration of the process $X$.
Then, for every $x\in \R^d\setminus \N$ and $k\ge 1$, by the Markov property and Theorem \ref{T:lower},
\begin{align*}\Pp^x(B_k|{F}_{t_k})\ge&\min_{z\in \R^d\setminus \N}\Pp^z(|X_{t_k}-z|\ge c\psi(t_{k-1}))\\
\ge &\min_{z\in \R^d\setminus \N}\int_{\{c\psi(t_{k-1})\le |y-z|\le t_{k}\}} p(t_k,z,y)\,\d y\\
\ge& c_1t_{k}^{-d/2}\min_{z\in \R^d\setminus \N}\int_{\{c\psi(t_{k-1})\le |y-z|\le t_{k}\}} \exp\left(-\frac{c_2|z-y|^2}{t_{k}}\right)\,\d y\\
\ge&c_3 \int_{c^2\log\log (t_{k-1})/2}^{t_k} e^{-c_2s} s^{d/2-1}\,\d s\\
\ge& c_4 k^{-c^2c_2}.
\end{align*}
Choosing $c>0$ small enough such that $c^2c_2\in(0,1]$, we have
$$\sum_{k=1}^\infty \Pp^x(B_k|{F}_{t_k})=\infty.$$
Then by the second Borel-Cantelli lemma,
$$\Pp^x(\limsup B_k)=1.$$
This yields the desired assertion, see e.g. the proof of \cite[Theorem 3.1(2)]{SW17}.
\end{proof}

\noindent \textbf{Acknowledgements.}
The research of Yuichi Shiozawa is supported
by JSPS KAKENHI Grant Number JP26400135, JP17K05299,
and Research Institute for Mathematical Sciences, Kyoto University.
The research of Jian Wang is supported by National
Natural Science Foundation of China (No.\ 11522106), the Fok Ying Tung
Education Foundation (No.\ 151002), National Science Foundation of
Fujian Province (No.\ 2015J01003), the Program for Probability and Statistics: Theory and Application (No. IRTL1704), and Fujian Provincial
Key Laboratory of Mathematical Analysis and its Applications
(FJKLMAA).

\address{
Yuichi Shiozawa\\
Department of Mathematics\\
Graduate School of Science\\
Osaka University \\
Toyonaka, Osaka, 560-0043,
Japan
}
{\texttt{shiozawa@math.sci.osaka-u.ac.jp}}

\address{
Jian Wang\\
College of Mathematics and Informatics \& Fujian Key Laboratory of Mathematical Analysis and Applications (FJKLMAA)\\
Fujian Normal University\\
350007 Fuzhou, P.R. China.
}
{\texttt{jianwang@fjnu.edu.cn}}
\end{document}